\documentclass[11pt, reqno]{amsart}
\pdfoutput=1
\usepackage{amssymb,amsmath,latexsym}
\usepackage{epsfig}
\usepackage{mathrsfs}
\usepackage{color}
\usepackage{graphicx}
\usepackage{epic,eepic,wrapfig,color, ifthen}
\usepackage{url}
\usepackage[colorlinks=true, citecolor=blue, filecolor=cyan, linkcolor=red, urlcolor=magenta]{hyperref}

\usepackage{upgreek}
\usepackage{amsthm}
\usepackage{amsfonts}
\usepackage{pmboxdraw}
\usepackage{verbatim}
\usepackage{color, mathtools, textcomp}
\usepackage{mathtools}
\usepackage{kotex}
\usepackage{cite}
\usepackage[normalem]{ulem}
\usepackage{enumerate}
\mathtoolsset{showonlyrefs}
\usepackage[bottom]{footmisc}

\newcommand{\cal}[1]{\mathcal{#1}}

\newcommand\ep{{\mathsf{d}}}
\newcommand\ups{{\upsilon}}

\newcommand\1{{\bf 1}}
\newcommand\nn{\nonumber}

 \newcommand\sB {{\cal B}} 
\newcommand\sD {{\cal D}}
\newcommand\sE {{\cal E}} 
\newcommand\sF {{\cal F}}
 \newcommand\sH {{\cal H}} 
 
\newcommand\sL {{\cal L}}

 \newcommand\sU {{\cal U}}

 \newcommand\Z {{\mathbb Z}}
\newcommand\E {{\mathbb E}} 

\newcommand\U{{\mathrm{U}}}

\newcommand \VRDo {${\mathrm {VRD}}_{R_0'}(U)$}
\newcommand \VRDi {${\mathrm {VRD}}^{R_\infty'}(\ups)$}

\newcommand \Taili {$\mathrm{Tail}^{R_\infty}( \varphi,\ups)$}

\newcommand \Ei {$\mathrm{E}^{R_\infty}(\varphi,\ups)$}

\newcommand \NDLi {${\mathrm{NDL}}^{R_\infty}(\varphi,\ups)$}

\newcommand \Tailil {$\mathrm{Tail}^{R_\infty}(\varphi,\ups,\le)$}
\newcommand \Tailig {$\mathrm{Tail}^{R_\infty}(\varphi,\ups,\ge)$}

\newcommand\la{{\langle}}
\newcommand\ra{{\rangle}}

\newcommand\FF{{\mathcal F}}
\newcommand\EE{{\mathcal E}}

\newcommand\N{{\mathbb N}}
\newcommand\bL{{\mathbb L}}
\newcommand\eps{\varepsilon}

\newcommand\wt{\widetilde}

\theoremstyle{plain}
\newtheorem{thm}{Theorem}[section]
\newtheorem{lem}[thm]{Lemma}
\newtheorem{cor}[thm]{Corollary}
\newtheorem{remark}[thm]{Remark}
\newtheorem{prop}[thm]{Proposition}
\newtheorem{defn}[thm]{Definition}
\newtheorem{definition}[thm]{Definition}

\newtheorem{example}[thm]{Example}

\theoremstyle{definition}
\newtheorem*{eg*}{Example}
\newtheorem*{egs*}{Examples}
\newtheorem*{def*}{Definition}

\begin{document}
	\title[General Law of iterated logarithm: Liminf laws]{
	General Law of iterated logarithm for Markov processes: Liminf laws}	
	
	\author{Soobin Cho, Panki Kim and  Jaehun Lee}\thanks{This work was supported by the National Research Foundation of Korea(NRF) grant funded by the Korea government(MSIP) 
(No. NRF-2021R1A4A1027378).
}

	\address[Cho]{Department of Mathematical Sciences,
		Seoul National University,
		Seoul 08826, Republic of Korea}
	\curraddr{}
	\email{soobin15@snu.ac.kr}
	
	\address[Kim]{Department of Mathematical Sciences and Research Institute of Mathematics,
		Seoul National University,
		Seoul 08826, Republic of Korea}

	\curraddr{}
	\email{pkim@snu.ac.kr}
	
	\address[Lee]{
	Korea Institute for Advanced Study,  
	Seoul 02455,
	Republic of Korea}

	\curraddr{}
	\email{hun618@kias.re.kr}


	\maketitle
		
\begin{abstract}

Continuing from \cite{CKL}, in this paper, we discuss general criteria and forms of liminf  laws of iterated logarithm (LIL) for continuous-time Markov processes. Under some minimal assumptions, which are weaker than those in \cite{CKL}, we establish  liminf LIL at zero (at infinity, respectively) in general metric measure spaces. In particular, our assumptions for   liminf  law of LIL at zero and the form of liminf LIL are truly local so that we can cover highly space-inhomogenous cases. Our results cover  all examples in \cite{CKL} including random conductance models with long range jumps. Moreover, we show that the general form of  liminf  law of LIL at zero holds for a large class of jump processes whose jumping measures have logarithmic tails and Feller processes with symbols of varying order which are not covered before.

\medskip


\end{abstract}
\allowdisplaybreaks

\section{Introduction and general result}\label{s:intro}

Let $Y:=(Y_t)_{t\ge0}$ be a nontrivial strictly $\beta$-stable process on ${\mathbb R}^d$ with $0<\beta \le 2$, in the sense of \cite[Definition 13.1]{Sa13}. 
Assume that none of the  one dimensional projections of  $Y$ is a subordinator, and $Y$ has no drift when $\beta=1$ (namely, $\tau=0$ in \cite[(14.16)]{Sa13}). Then $Y$  satisfies the following Chung-type  liminf LIL: There exists a constants $C \in (0,\infty)$ such that 
\begin{align}
\label{e:ieq0}
\liminf_{t \to 0 \;\; (\text{resp. }  t\to \infty) } \frac{\sup_{0<s\le t} |Y_s|}{(t/\log|\log t|)^{1/\beta}}
=C \;\; \mbox{ a.s.} 
\end{align}
See, e.g. \cite[Chapters 47, 48]{Sa13}.
 
The liminf LIL \eqref{e:ieq0} was  established for random walks on $\Z$ by Chung \cite{Ch48} under  the  assumption that  their i.i.d. increments have a finite third moment and expectation zero. 
The liminf LIL in \cite{Ch48} 
 was improved to a finite second moment assumption by Jain and Pruitt \cite{JP75}. 
For some related results, we refer to  \cite{Wi74, EM94, Ke97}. 
Chung also showed the large time result of \eqref{e:ieq0} for a Brownian motion in ${\mathbb R}$. The liminf LIL has been  extended to non-Cauchy $\beta$-stable processes on ${\mathbb R}^d$ with $\beta<d$ by  Taylor \cite{T}, increasing random walks and subordinators by Fristedt and Pruitt \cite{FP71}, and symmetric L\'evy processes in ${\mathbb R}$ by Dupuis \cite{Du74}. Then Wee \cite{Wee88} succeeded in obtain liminf LILs for numerous non-symmetric
L\'evy processes in ${\mathbb R}$. See also \cite{ADS13, BM} and the references therein. Recently, Knopova and Schilling \cite{KS} extended liminf LIL  at zero to non-symmetric L\'evy-type  processes in ${\mathbb R}$. Also, very recently, the second named author, jointly with Kumagai and Wang  \cite{KKW17} extended  liminf LIL for symmetric mixed stable-like Feller processes on metric measure spaces.	

The purpose of this paper is to understand asymptotic behaviors of a given Markov process by establishing liminf  law  of iterated logarithms for both near zero and near infinity under some minimal assumptions. In particular, we introduce new but general version of it. See Theorem \ref{inf} below. 
Our assumptions are  weak enough so that our results cover a lot of Markov processes including jump processes with diffusion part,  jump processes with small jumps of slowly varying intensity, 
some  non-symmetric processes,  processes with singular jumping kernels and random conductance models with long range jumps. See the examples in Sections \ref{s:Feller}-\ref{s:hunt}, and the references therein.  In particular,  the class of Markov processes considered in this paper extends the results of \cite{KKW17}. 
Moreover, metric measure spaces in this paper can be random, disconnected and highly space-inhomogeneous (see Definition \ref{d:VD}).

\medskip

Throughout  this section,  Section \ref{s:proof} and Appendix \ref{s:A}, 
 we assume that $(M,d,\mu)$ is a locally compact separable metric measure space where $\mu$ is a positive Radon measure on $M$ with full support. 
We add a cemetery point $\partial$ to $M$ and denote $M_\partial=M \cup \{\partial\}$.   We consider  a  Borel standard Markov process $X = (\Omega, \sF_t, X_t, \theta_t, t \ge 0;  {\mathbb P}^x, x \in M_\partial)$  on $M_\partial$ with the lifetime $\zeta:=\inf\{t>0:X_t = \partial\}$. Here $(\theta_t)_{t\ge0}$ is the shift operator with respect to  $X$ which is defined as 
$X_s ( \theta_t \, \omega) = X_{s+t}(\omega)$
 for all $t,s \ge 0$. 
Since $X$ is a Borel standard process, $X$  has a L\'evy system in the sense of  \cite[Theorem 1.1]{BJ73}. In  this paper, we always assume that $X$ admits a   L\'evy system of the form $(J(x,\cdot),ds)$ so that 
for any  $z \in M$, $t>0$ and nonnegative Borel function $F$ on $M \times M_\partial$ vanishing on the diagonal,
\begin{equation*}
	\E^z \bigg[ \,\sum_{s \le t} F(X_{s-}, X_s)\,  \bigg] = \E^z \left[ \int_0^t \int_{M_\partial} F(X_s,y)J(X_s,dy)ds \right].
\end{equation*}
The measure $J(x,dy)$ on $M_\partial$ is called the \textit{L\'evy measure} of the process $X$. Here we emphasize that the killing term $J(x, \partial)$ is included in the L\'evy measure.

For $x \in M$ and $r\in(0,\infty]$, set $B(x,r):= \{ y \in M : d(x,y) <r  \}$  and  $V(x,r):=\mu(B(x,r))$  with a convention $B(x,\infty)=M$. For a subset $U \subset M$, we denote $\updelta_U(x)$ for the distance between $x$ and $M \setminus U$, namely, 
\begin{equation*}
	\updelta_U(x)=\inf \{ d(x,y) : y \in M \setminus U\}, \quad x \in M.
\end{equation*}
We fix a base point $o \in M$ and define
\begin{equation*}
	{\mathsf{d}} (x) = d(x,o) + 1, \quad x \in M.
\end{equation*}
Since ${\mathsf{d}} (x) \ge 1$, the map 
$\to$
${\mathsf{d}}(x)^\ups$
 is nondecreasing on $(0,\infty)$. 

For a Borel set $D\subset M$, we denote
$$\tau_D:= \inf \{ t>0 : X_t \in M_\partial \setminus D \}$$ for  the first exit time of $X$  from $D$.

We are now ready to introduce our assumptions. Our assumptions are given in terms of mean exit times, tails of L\'evy measures and survival probabilities on balls. Our assumptions are weaker than those in \cite{CKL}, see Lemmas \ref{l:NDL} and \ref{l:hkendl} in Appendix \ref{s:A}.

Here are our assumptions for liminf LIL at zero.  Let $U\subset M$ be an open subset of $M$.

\bigskip

\setlength{\leftskip}{5mm}
\textit{There exist constants $R_0>0$, $C_0\in(0,1)$, $C_1>1$, $C_i>0$, $2\le i\le 7$  such that for every $x \in U$ and  $0<r<R_0 \wedge (C_0\updelta_U(x))$,}

\begin{equation}\label{A1}\tag{A1}
C_1^{-1} \E^y[\tau_{B(y,r)}]\le  \E^x[\tau_{B(x,r)}] \le C_1 \E^y[\tau_{B(y,r)}]\quad \text{for all} \; y \in B(x,r),
\end{equation}\\[-6mm]
\begin{equation}\label{A2}\tag{A2}
\lim_{r \to 0} \E^x[\tau_{B(x,r)}] = 0, \quad  \E^x[\tau_{B(x,r)}] \le C_2 \E^x[\tau_{B(x,r/2)}],
\end{equation}\\[-6mm]
\begin{equation}\label{A3}\tag{A3}
 J(x,M_\partial \setminus B(x,r)) \le  \frac{C_3}{ \E^x[\tau_{B(x,r)}]},
\end{equation}\\[-5mm]
\begin{equation}\label{A4}\tag{A4}
C_4e^{-C_5 n} \le {\mathbb P}^x\big(\tau_{B(x,r)} \ge n \E^x[\tau_{B(x,r)}]\big) \le C_6e^{-C_7 n} \quad \text{for all} \; n \ge 1.
\end{equation}

\bigskip

\setlength{\leftskip}{0mm}

Next, we give assumptions for liminf LIL at infinity. 

\bigskip

\setlength{\leftskip}{5mm}
\textit{There exist constants $R_\infty\ge 1$, $\ups\in(0,1)$, $\ell,C_1>1$,  $C_i>0$, $2\le i\le 7$  such that for every $x \in M$ and  $r>R_\infty {\mathsf{d}}(x)^\ups$,}

\begin{equation}\label{B1}\tag{B1}
C_1^{-1}\E^o[\tau_{B(o,r)}] \le	\E^x[\tau_{B(x,r)}] \le C_1	 \E^o[\tau_{B(o,r)}],
\end{equation}\\[-6mm]
\begin{equation}\label{B2}\tag{B2}
2 \E^o[\tau_{B(o,s/\ell)}] \le 	\E^o[\tau_{B(o,s)}] \le C_2	\E^o[\tau_{B(o,s/2)}]  \quad \text{for all} \; s >R_\infty,
\end{equation}\\[-6mm]
\begin{equation}\label{B3}\tag{B3}
J(x,M_\partial \setminus B(x,r)) \le  \frac{C_3}{\E^x[\tau_{B(x,r)}]},
\end{equation}\\[-5mm]
\begin{equation}\label{B4}\tag{B4}
	C_4e^{-C_5 n} \le {\mathbb P}^x\big(\tau_{B(x,r)} \ge n\E^x[\tau_{B(x,r)}]\big) \le C_6e^{-C_7 n} \quad \text{for all} \; n \ge 1.
\end{equation}

\medskip

We recall the following figure from \cite[Figure 1]{CKL}  which shows ranges of $r$ in our conditions.
\begin{figure}[h!]\label{f:1}
	\includegraphics[width=0.36\columnwidth]{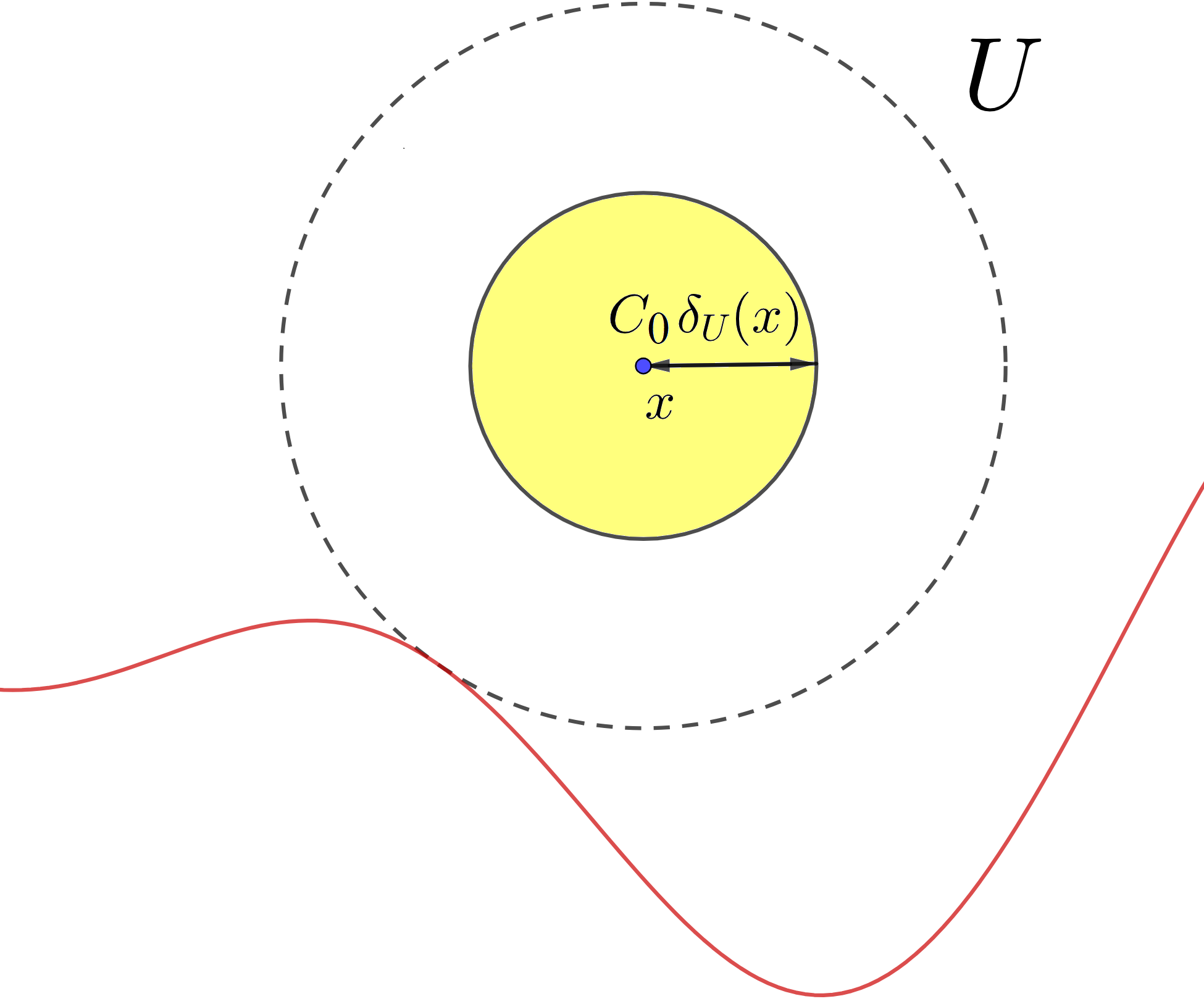}  \hspace{20mm}	\includegraphics[width=0.32
	\columnwidth]{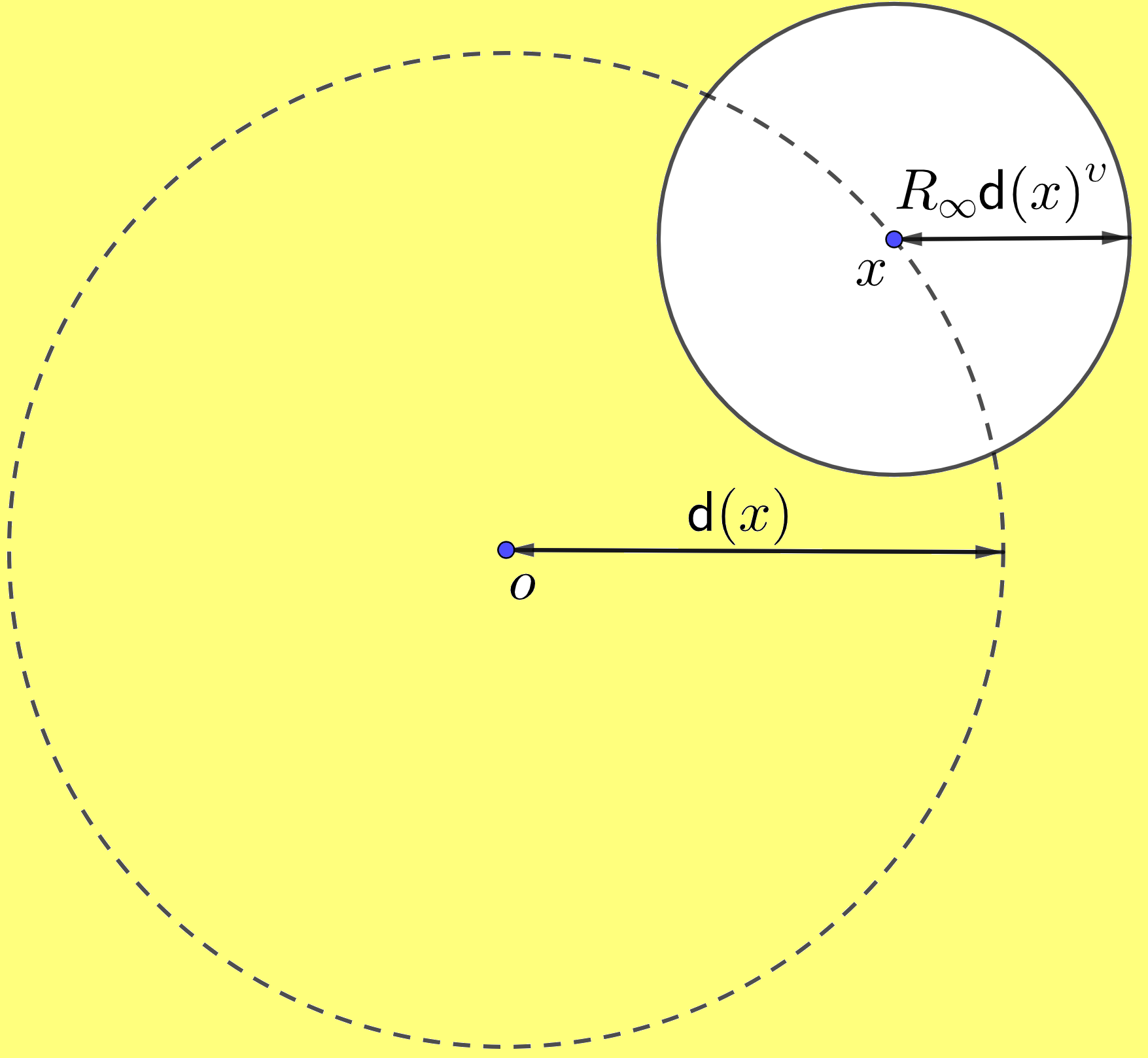} \vspace{-2mm}
	\caption{Range of $r$ in local conditions}
\end{figure}

\setlength{\leftskip}{0mm}

\begin{remark}\label{r:basic}
	{\rm (i) Note that, like \cite{CKL},  we impose conditions at infinity \eqref{B1}-\eqref{B4}  only for $r>R_\infty {\mathsf{d}}(x)^\ups$. By considering such weak assumptions at infinity, our LILs cover  some random conductance models. See Section \ref{s:RCM} below.
	
	\noindent (ii) The assumptions \eqref{A3} and \eqref{B3} are quite mild and  natural. 
	 Let $x \in U$. 
	 	Suppose  \eqref{A2} holds for $U$ and for $0<r<R_0 \wedge (C_0\updelta_U(x))$
	$$J(y,M_\partial \setminus B(x,r)) \ge c J(x,M_\partial \setminus B(x,r)) \quad \text{ for all } y \in B(x, r/2).$$
	Then, by the L\'evy system, we have that for  $0<r<R_0 \wedge (C_0\updelta_U(x))$
\begin{align*}
1 &\ge {\mathbb P}^x(X_{\tau_{B(x,r/2)}} \in M_\partial \setminus B(x,r) ) = \E^x \left[\int_0^{\tau_{B(x,r/2)}} \int_{M_\partial \setminus B(x,r)} J(X_s,dy) ds \right] \\  &\ge
c \,\E^x[\tau_{B(x,r/2)}] J(x,M_\partial \setminus B(x,r)) \ge c \, C_2^{-1} \E^x[\tau_{B(x,r)}] J(x,M_\partial \setminus B(x,r)).
\end{align*}

	\noindent (iii) The lower inequality in \eqref{B2} implies that $\lim_{r \to \infty} \E^o[\tau_{B(o,r)}] = \infty$. Hence, under \eqref{B1} and \eqref{B2}, we have $\lim_{r \to \infty} \E^x[\tau_{B(x,r)}] = \infty$ for all $x \in M$.

	\noindent (iv) We need the lower inequality in \eqref{B2} because the  space $M$ may be highly inhomogeneous. This inequality  implies that the map $r \mapsto \E^o[\tau_{B(x,r)}]$ grows at least polynomially for large values of $r$. Hence under this condition, with a high probability,  $X_t$ is not located in regions extremely far away from the base point $o$ in view of the time $t$ where we do not impose any conditions.

}
\end{remark}

From now on, {\it whenever condition \eqref{A1}-\eqref{A3} are assumed with $R_0>0$  for an open set $U$, we let $\phi(x,r)$ be 
any increasing function 
defined on $U \times (0,R_0)$ which satisfies the following condition: There is a  constant $C \ge 1$ such that for all $x \in U$ and $0<r<R_0 \wedge (C_0\updelta_U(x))$,
\begin{equation}\label{e:phi}
 C^{-1}\E^x[\tau_{B(x,r)}] \le \phi(x,r) \le C\E^x[\tau_{B(x,r)}].
\end{equation} }
Also, {\it whenever condition \eqref{B1}-\eqref{B3} are assumed with $R_\infty \ge 1$ and $\ups \in (0,1)$, we let $\phi(r)$ be any function on $M \times (0,\infty)$ satisfying \eqref{e:phi} for all $x \in M$ and $r>R_\infty {\mathsf{d}}(x)^\ups$, with $\phi(r)$ instead of $\phi(x,r)$. }

Then by condition \eqref{A2}, we see that there exist constants $\beta_2>0$ and $C_U\ge 1$ such that
\begin{equation}\label{e:scaling-zero}
	\frac{\phi(x,r)}{\phi(x,s)} \le C_U \bigg(\frac{r}{s}\bigg)^{\beta_2} \quad \text{for all} \; x \in U, \; 0<s\le r<R_0 \wedge (C_0\updelta_U(x)),
\end{equation}
and by \eqref{B2},  there exist constants $\beta_1,\beta_2>0$ and $C_U\ge 1 \ge C_L >0$ such that
\begin{equation}\label{e:scaling-infty}
C_L \bigg(\frac{r}{s}\bigg)^{\beta_1} \le 	\frac{\phi(r)}{\phi(s)} \le C_U \bigg(\frac{r}{s}\bigg)^{\beta_2} \quad \text{for all} \;  r \ge s> R_\infty {\mathsf{d}}(x)^\ups.
\end{equation}

Now, we give our results in full generality. Our first result is the liminf  law of LIL at zero. Note that we don't put any extra assumptions on our metric measure space such as volume doubling property.

\begin{thm}\label{inf}
Suppose that \eqref{A1}-\eqref{A4} hold for an open subset $U \subset M$. Then, there are constants $a_2 \ge a_1>0$ such that for all $x \in U$, there exists a constant $a_x \in [a_1,a_2]$ satisfying
	\begin{equation}\label{1}
	\liminf_{t\to0} \frac{\phi\big(x,\sup_{0<s\le t} d(x, X_s)\big)}{t/\log|\log t|}=a_x,\qquad
	{\mathbb P}^x\mbox{-a.s.}
	\end{equation}
\end{thm}

Note that in Theorem \ref{inf}, $r \mapsto \phi(x,r)$ for $x \in U$ can be slowly varying at zero so that we can cover jump processes whose jumping measures have logarithmic tails. 
When $\phi(x,r)$ is a mixed polynomial type near zero (i.e., both \eqref{e:scaling-zero} and \eqref{e:scaling-zero-lower} hold), we recover the classical form of the liminf LIL at zero.

We denote  $\phi^{-1}(x,t):=\sup\{r>0:\phi(x,r)\le t\}$ for the right continuous inverse of $\phi(x,\cdot)$.

\begin{cor}\label{c:inf}
Suppose that \eqref{A1}-\eqref{A4} hold for an open subset $U \subset M$ and there exist constants $\beta_1,C_L>0$ such that 
\begin{equation}\label{e:scaling-zero-lower}
	\frac{\phi(x,r)}{\phi(x,s)} \ge C_L \bigg(\frac{r}{s}\bigg)^{\beta_1} \quad \text{for all} \; x \in U, \; 0<s\le r<R_0 \wedge (C_0\updelta_U(x)).
\end{equation}
 Then, there are constants $\wt a_2 \ge \wt a_1>0$  such that for all $x \in U$, there exists a constant $\wt a_x \in [\wt a_1,\wt a_2]$ satisfying
	\begin{equation}\label{c1}
	\liminf_{t\to0} \frac{\sup_{0<s\le t} d(x, X_s)}{\phi^{-1}(x,t/\log|\log t|)}=\wt a_x,\qquad
	{\mathbb P}^x\mbox{-a.s.}
	\end{equation}
\end{cor}

\bigskip
Our second result is the liminf  law of LIL at infinity.  
\begin{thm}\label{t:infinf}
	Suppose that \eqref{B1}-\eqref{B4} hold. Then, there are constants $b_2 \ge b_1>0$ such that
	\begin{equation}\label{20}
	\liminf_{t\to \infty} \frac{\sup_{0<s\le t} d(x, X_s)}{\phi^{-1}(t/\log\log t)} \in [b_1, b_2],\qquad
	{\mathbb P}^y\mbox{-a.s.}  \;\; \forall x,y \in M.
	\end{equation}
\end{thm}

The $\liminf$ in \eqref{20} may not be deterministic. In \cite{CKL}, we have obtained a zero-one law
 for shift-invariant events
 under volume of doubling  assumptions 
  and near diagonal lower estimates on heat kernel (see Proposition \ref{p:law01} below).
Using the same zero-one law, in this paper we also establish the  deterministic limit in liminf law.

 We recall following versions of volume doubling  property from authors' previous paper \cite{CKL}.

\begin{defn}\label{d:VD}
	{\rm (i) For an open set $U \subset M$ and $R_0' \in (0,\infty]$, we say that \textit{the  interior volume doubling and reverse doubling property} \VRDo \  holds  if there exist  constants $C_V \in (0,1)$, $d_2 \ge d_1>0$ and $C_\mu \ge c_\mu>0$ such that for all $x \in U$ and $0<s\le r<R_0' \wedge (C_V \updelta_U(x))$,
		\begin{equation}\label{e:VRD}
			c_\mu \left( \frac{r}{s} \right)^{d_1} \le \frac{V(x,r)}{V(x,s)} \le  C_\mu \left( \frac{r}{s} \right)^{d_2}.
		\end{equation}
		
		\noindent (ii) 	For  $R_\infty'\ge1$  and $\ups \in (0,1)$, we say that \textit{a weak volume doubling and reverse doubling property at infinity}  \VRDi \ holds  if there exist constants $d_2 \ge d_1>0$ and $C_\mu \ge c_\mu>0$   such that  \eqref{e:VRD} holds for all $x \in M$ and $r \ge s > R_\infty' {\mathsf{d}}(x)^\ups$.  
} \end{defn}

For an open set $D\subset M$, let  $X^D$ be the part process  of $X$ defined as  $X_t^D := X_t \, \1_{\{\tau_D > t\}} + \partial \, \1_{\{\tau_D \le t\}}$. Then $X_t$ is  a Borel standard process on $D$. See, e.g. \cite[Section 3.3]{CF}.
Let $(P^D_t)_{t \ge 0}$ be the semigroup associated with $X^D$,
namely, $P^D_t f(x) := \E^x [f(X^D_t)]$. We call  a Borel  measurable  function $p^D: (0,\infty) \times D \times D \to [0,\infty]$  the  \textit{heat kernel} (transition density) of  $(P^D_t)_{t \ge 0}$ (or  $X^D$) if the followings hold: 

\medskip

\setlength{\leftskip}{4mm}
\noindent (i) $P^D_t f(x) = \int_{D} p^D(t,x,y)f(y) \mu(dy)$ for all $t>0$, $x \in D$ and $f \in L^\infty(D;\mu)$.

\smallskip

\noindent (ii) $p^D(t+s,x,y) = \int_{D} p^D(t,x,z) p^D(s,z,y) \mu(dz)$ for all $t,s>0$ and $x,y \in D$.

\medskip

\setlength{\leftskip}{0mm}

\noindent 
We simply write $P_t$ for $P_t^M$ and  $p(t,x,y)$ for $p^M(t,x,y)$.   

We now consider the following near diagonal lower estimates on heat kernels:
\begin{align*}
	&\textit{There exist constants $R_\infty' \ge 1$, $\ups,\eta \in (0,1)$, $c_*>0$ such that for all $x \in M$ and}\nn\\
	&\textit{$r>R_\infty' {\mathsf{d}}(x)^\ups$, the heat kernel $p^{B(x,r)}(t,x,y)$ of $X^{B(x,r)}$ exists and} 
\end{align*}
\begin{equation}\label{B4+}
	p^{B(x,r)}( \phi(\eta r),y,z) \ge \frac{c_*}{V(x, r)} \quad \text{for all} \;\,  y,z \in B(x, \eta^2r).\tag{B4+} 
\end{equation}

\medskip

Under \eqref{B1} and \VRDi, the above condition \eqref{B4+} is stronger than \eqref{B4}. See Proposition \ref{r:E} below.

\begin{remark} \label{r:NDL} {\rm
The standard version of near diagonal lower estimates on heat kernels (without the restriction $r>R_\infty {\mathsf{d}}(x)^\ups$) has been studied a lot.
In particular, in \cite{CKW16b} and \cite{CKW19}, it is shown that, for a large class of symmetric Hunt process, the standard version of near diagonal lower estimates on heat kernels can be obtained under  \eqref{B2} and a H\"older-type  regularity of corresponding harmonic functions
See. \cite[Proposition 4.9]{CKW16b} and its proof.}

\end{remark}

\begin{cor}\label{c:infinf}
	Suppose that \VRDi \ holds. If \eqref{B1}, \eqref{B2}, \eqref{B3} and \eqref{B4+} hold, then there exists a constant $b_\infty \in (0, \infty)$ such that 
	\begin{equation}\label{2}
	\liminf_{t\to \infty} \frac{\sup_{0<s\le t} d(x, X_s)}{\phi^{-1}(t/\log\log t)} = b_\infty,\qquad
	{\mathbb P}^y\mbox{-a.s.} \;\; \forall x,y \in M.
	\end{equation}
\end{cor}

\begin{remark}\label{r:a}
{\rm  (i)	
 Once we prove that the liminf LILs \eqref{1} and \eqref{c1} hold true with $\phi(x,r)=\E^x[\tau_{B(x,r)}]$, by the Blumenthal's zero-one law, they hold true with general $\phi$ satisfying \eqref{e:phi} after redefining constants $a_1,a_2, \wt a_1, \wt a_2$ by $C^{-1}a_1,Ca_2, C^{-1}\wt a_1, C\,\wt a_2$, respectively, with the constant $C \ge 1$ in \eqref{e:phi}. Similarly, thanks to the zero-one law given in Proposition \ref{p:law01}, it suffices to prove Theorem \ref{t:infinf} and Corollary \ref{c:infinf} with a particular function  $\phi(r):=\E^o[\tau_{B(o,r)}]$.

\noindent   (ii) Using the zero-one law in Proposition \ref{p:law01} again, we see that the liminf LIL \eqref{2} remains true even if  the function $\phi$, which comes from condition \eqref{B4+}, is replaced by any function $\wt \phi$ comparable to $\phi$.} 
\end{remark}

 Let us also consider the following 
 counterpart of \eqref{B4+}:
\begin{align*}
	&\textit{For a given open set $U \subset M$, there exist constants $R_0'>0$, $C_0',\eta  \in (0,1)$ and $c_*>0$ such that}\nn\\
	&\textit{for all $x \in U$ and $0<r<R_0' \wedge (C_0' \updelta_U(x))$, the heat kernel $p^{B(x,r)}(t,x,y)$ of $X^{B(x,r)}$ exists and}
\end{align*}
\begin{equation}\label{A4+}
	p^{B(x,r)}(\phi(x,\eta r),y,z) \ge \frac{c_*}{V(x, r)} \quad \text{for all} \;\,  y,z \in B(x, \eta^2r).\tag{A4+}
\end{equation}

\begin{prop}\label{r:E}
	(i) Suppose that \VRDo \ holds. If \eqref{A2} and \eqref{A4+} hold, then  \eqref{A4} holds  with some $R_0>0$ and $C_0\in (0,1)$.
	
	\noindent	(ii) Suppose that \VRDi \ holds. If \eqref{B2} and \eqref{B4+} hold, then  \eqref{B4} holds  with some $R_\infty\ge1$.
\end{prop}
\begin{proof} (i) By following the proof of \cite[Proposition 4.3(i)]{CKL} and using $\phi(x,r)$ instead of $\phi(r)$ therein, we can deduce that there exist constants $R_1,c_2,c_3>0$, $c_1>1$ and $C_0 \in (0,1)$ such that for all $x \in U$, $0<c_1r<R_1 \wedge (C_0 \updelta_U(x))$ and $n \ge 1$,
\begin{equation*}
	e^{-c_2n} \le {\mathbb P}^x\big(\tau_{B(x,r)} \ge n\E^x[\tau_{B(x,c_1r)}](x,c_1r)\big) \le e^{-c_3 n}.
\end{equation*}
By taking $R_1$ small enough if needed, we may assume that \eqref{A2} holds with $R_0=R_1$. Using \eqref{e:scaling-zero},  it follows that  for all $x \in U$, $0<r<R_1 \wedge (C_0 \updelta_U(x))$ and $n \ge 1$,
\begin{align*}
	e^{-c_2n} &\le {\mathbb P}^x\big(\tau_{B(x,r)} \ge n\phi(x,c_1r)\big)  \le 	{\mathbb P}^x\big(\tau_{B(x,r)} \ge n\phi(x,r)\big)\\
	& \le   {\mathbb P}^x\big(\tau_{B(x,r)} \ge c_1^{-\beta_2}C_U^{-1}n\phi(x,c_1r)\big) \le e^{-c_3( c_1^{-\beta_2}C_U^{-1}n -1 )}.
\end{align*}
(ii) Analogously, we can deduce the result by following the proof of  \cite[Proposition 4.3(ii)]{CKL}. \end{proof}

	The rest of the paper is organized as follows. In Sections \ref{s:Feller}--\ref{s:hunt}, we 
	show that the conditions 
	\eqref{A1}-\eqref{A4}, \eqref{B1}-\eqref{B4}  and \eqref{B4+} 
	can be checked  for  important classes of Markov jump processes, which may be non-symmetric and space-inhomogeneous. Thus we can
	apply our main theorems to get  explicit liminf LILs	for them.

	 Precisely, in  Section \ref{s:Feller}, we consider general (non-symmetric) Feller processes on ${\mathbb R}^d$ whose domain of the generator  contains $C_c^\infty({\mathbb R}^d)$ 
	 and introduce some local assumptions (see \eqref{O1}-\eqref{O4} below). 
	 Under these assumptions, we establish liminf LIL at zero for Feller processes on ${\mathbb R}^d$.
	 Then, combining results in this paper and \cite{CKL}, we present concrete examples of  non-symmetric Feller processes and Feller processes with singular L\'evy measures
	for which  both liminf LILs and limsup LILs hold.
  In the remainder of Section \ref{s:Feller}, we give another assumption (see {\bf (S)} below), which can be  checked directly from the symbols of Feller processes.  As a consequence, we show that  liminf LILs at zero  holds for Feller processes with symbols of varying order.
	 
Section \ref{s:RCM} revisits \cite[Section 3]{CKL} and  discuss liminf LILs for the random conductance model with long range jumps studied in \cite{CKW18, CKW20-1}. 
Our conditions at infinity is motivated by 	the random conductance model therein.
	 
  In Section \ref{s:hunt}, we deal with subordinate processes and symmetric Hunt processes whose tail of the L\'evy measure decays in (mixed) polynomial order. We assume that there is a Hunt process $Z$  enjoying sub-Gaussian heat kernel estimates. Then we show  that general liminf LIL holds true for  every subordinate process of $Z$ if the corresponding subordinator is not a  compound Poisson process. In particular, we get liminf LILs for jump processes with low intensity of small jumps such as geometric stable processes. Using local stability theorems obtained in \cite{CKL}, we also get liminf LIL for symmetric Hunt processes associated with a regular Dirichlet form.

  Section \ref{s:proof} is devoted to the proofs of our main theorems.  We follow the well-known arguments in 
  \cite[Chapter 3]{Du74},   \cite[Theorem 2]{KS}  and   \cite[Theorem 3.7]{KKW17}. 
  But non-trivial modifications are required since we allow our processes and state spaces to be highly  space-inhomogeneous.
 The paper ends with Appendix \ref{s:A} which contains some comparisons between conditions of the current paper and \cite{CKL} and a simple lemma about lower heat kernel estimates for Dirichlet heat kernel.

\medskip

{\bf Notations}: Values of capital letters with subscripts $C_i$, $i=0,1,2,...$ are fixed throughout the paper
 both at zero and at infinity. 
Lower case letters with subscripts $a_i$, $c_i$, $i=0,1,2,...$  denote positive real constants  and are fixed in each statement and proof, and the labeling of these constants starts anew in each proof. 
We use the symbol ``$:=$'' to denote a definition, 
which is read as ``is defined to be.''  
Recall that $a\wedge b:=\min\{a,b\}$ and  $a\vee b:=\max\{a,b\}$. We  denote by $\overline{A}$ the closure of $A$. We extend a function $f$ defined on $M$ to $M_\partial$ by setting $f(\partial)=0$. The notation $f(x) \asymp g(x)$ means that there exist constants $c_2 \ge c_1>0$ such that $c_1g(x)\leq f (x)\leq c_2 g(x)$ for  specified range of $x$. 
For $D \subset M$, denote by $C_c(D)$ the space of all continuous functions  with compact support in $D$.

\section{LIL for Feller processes on ${\mathbb R}^d$}\label{s:Feller}

Throughout this section, we assume that $X$ is a  Feller process on ${\mathbb R}^d$ with the generator $(\sL, \sD(\sL))$ such that $C_c^\infty({\mathbb R}^d) \subset \sD(\sL)$. It is well known that the generator $\sL$ restricted to $C_c^\infty({\mathbb R}^d)$ is a \textit{pseudo-differential operator}, which has the following representation (see \cite{Co65}): 
\begin{equation*}
	\sL u(x)= - (2\pi)^{-d} \int_{{\mathbb R}^d} e^{i \la x,  \xi \ra}q(x, \xi) \int_{{\mathbb R}^d} e^{- i \la y, \xi\ra}u(y)dyd \xi, \quad  u \in C_c^\infty({\mathbb R}^d),
\end{equation*}
where the function  $q:{\mathbb R}^d \times {\mathbb R}^d \to \mathbb{C}$, which is called the {\it symbol} of $X$ (or $\sL$),  enjoys the following {L\'evy-Khinchine formula}
\begin{equation*}
	q(x, \xi) = {\mathbf c}(x) - i \la {\mathbf b}(x), \xi \ra +  \la \xi, {\mathbf a}(x)\xi \ra + \int_{{\mathbb R}^d \setminus \{0\}} (1- e^{i\la z, \xi \ra} + i \la z, \xi \ra \1_{\{| z| \le 1\}})\nu(x, dz).
\end{equation*}
Here $({\mathbf c}(x), {\mathbf b}(x), {\mathbf a}(x), \nu(x, dz))_{x \in {\mathbb R}^d}$ is a family of the L\'evy characteristics, that is, ${\mathbf c}:{\mathbb R}^d \to [0,\infty)$ and ${\mathbf b}:{\mathbb R}^d \to {\mathbb R}^d$ are measurable functions, ${\mathbf a}:{\mathbb R}^d \to {\mathbb R}^{d \times d}$ is a nonnegative definite matrix-valued function, and $\nu(x,dz)$ is a nonnegative, $\sigma$-finite kernel on ${\mathbb R}^d \times \sB({\mathbb R}^d \setminus \{0\})$ such that $\int_{{\mathbb R}^d \setminus \{0\}}(1 \land |z|^2) \nu(x, dz)<\infty$ for every $x \in {\mathbb R}^d$. 
{\it Throughout this section, we always assume that  ${\mathbf c}(x)$ is  identically zero} so that $X$ has no killing inside.

\medskip

Define for $x \in {\mathbb R}^d$ and $r>0$,
\begin{equation}\label{e:def-Phi}
\Phi(x,r) = \bigg(\sup_{|\xi| \le 1/r} \text{Re} \, q(x,\xi) \bigg)^{-1}.
\end{equation}
For example, if $q(x,\xi)=|\xi|^{\alpha(x)}(\log(1+|\xi|))^{\gamma(x)}$, then $\Phi(x,r)=r^{\alpha(x)} (\log ( 1+ 1/r))^{-\gamma(x)}$.

\medskip

Here are our assumptions on the Feller process $X$. Let $U \subset {\mathbb R}^d$ be an open subset.

\medskip

\setlength{\leftskip}{5mm}

\textit{There exist constants $R_0,  \eps_0>0$, $C_i\in(0,1)$, $8 \le i \le 11$ such that for every $x \in U$ and $\xi \in {\mathbb R}^d$ with $1/|\xi|<R_0 \wedge (C_8 \updelta_U(x))$}, the following hold
\begin{equation}\label{O1}\tag{O1}
 \lim_{r \to 0} \Phi(x,r)=0 \quad \text{(or, equivalently, either ${\mathbf a}(x) \neq 0$ or $\nu(x, {\mathbb R}^d)=\infty$)},
\end{equation}\\[-7mm]
\begin{equation}\label{O2}\tag{O2}
\sup_{|\xi'|\le |\xi|}\text{Re} \, q(x, \xi') \ge C_9 |  \, \text{Im} \, q(x, \xi) |,
\end{equation}\\[-7mm]
\begin{equation}\label{O3}\tag{O3}
\inf_{| x-y | \le 1/|\xi|} \textrm{Re}\,q(y,\xi) \ge C_{10} \sup_{| x-y|\le 1/|\xi|} \textrm{Re}\,q(y,\xi),
\end{equation}\\[-7mm]
\begin{equation}\label{O4}\tag{O4}
{\mathbb P}^x\big(2\la  X_t- x, z \ra \le -| X_t -x |  \big) \ge C_{11} \quad \text{for all} \;\;|z|=1 \; \text{and} \; 0<t<\eps_0\Phi\big(x, R_0 \wedge (C_8 \updelta_U(x))\big).
\end{equation}

\setlength{\leftskip}{0mm}

\medskip

  Under \eqref{O1}, the probability that the process $X$ starting from $x\in U$  stays at $x$ for a positive time is zero.
\eqref{O2} is not only a local formulation of \textit{the sector condition} but also a weaker version of it since we take the supremum in the left-hand side.  \eqref{O3} and \eqref{O4} give a weak spatial homogeneity of the process in $U$.  Similar conditions are appeared in \cite{KS} (see (A2)-(A3) therein) where small time Chung-type LILs for one-dimensional L\'evy-type processes were studied.  

\begin{remark}\label{r:O4}
	{\rm When $d=1$, \eqref{O4} is equivalent to  ${\mathbb P}^x(X_t-x \ge 0) \in [C_{11}, 1-C_{11}]$ for all $x \in U$ and $t \le \Phi\big(x, R_0 \wedge (C_8 \updelta_U(x))\big)$. Thus, if $d=1$ and $X$ is symmetric, then \eqref{O4} holds with $C_{11}=1/2$.
	}
\end{remark}

Following \cite{Pr81}, we set for $x \in {\mathbb R}^d$ and $r>0$,
\begin{equation}\label{e:def-h}
h(x,r)= \frac{1}{r^2} \lVert {\mathbf a}(x) \rVert +  \int_{{\mathbb R}^d} \left(\frac{ |z |^2}{r^2} \land 1 \right)  \nu(x, dz),
\end{equation}
where $\lVert {\mathbf a}(x)\rVert:=\sup_{\xi\in {\mathbb R}^d, \, |\xi|\le 1} \la \xi, {\mathbf a}(x)\xi\ra$. The function $h$ frequently appears in maximal inequalities for L\'evy-type processes.  See, e.g. \cite{Du74, KS}.

The following result is well-known. We give a full proof for the reader's convenience.

\begin{lem}\label{l:Phi-1}
	There exists a constant $C_{12}>1$ which depends only on the dimension $d$ such that 
	\begin{equation*}
	 \sup_{| \xi| \le 1/r} \text{\rm Re} \, q(x, \xi ) \le  2h(x,r) \le  C_{12} \sup_{1/(2r) \le | \xi| \le 1/r} \text{\rm Re} \, q(x, \xi ) \quad \text{for all} \;\; x \in {\mathbb R}^d, \; r>0.
	\end{equation*}
In particular, it holds that
\begin{equation}\label{e:Phi=h}
\frac{1}{2h(x,r)} \le 	\Phi(x,r) \le \frac{C_{12}}{2h(x,r)} \quad \text{for all} \;\; x \in {\mathbb R}^d, \; r>0.
\end{equation}
\end{lem}
\begin{proof} Using the inequality $1-\cos y \le y^2 \wedge 2$ for $y \in {\mathbb R}$, we get that for all $x \in {\mathbb R}^d$ and $r>0$,
\begin{equation*}
	\sup_{|\xi| \le 1/r}\text{Re}\,q(x, \xi) \le \frac{1}{r^2} \lVert {\mathbf a}(x) \rVert  + \sup_{|\xi| \le 1/r} \int_{{\mathbb R}^d } (1- \cos \la z, \xi \ra) \nu(x, dz) \le 2 h(x,r).
\end{equation*}

Next, it is clear that $\sup_{1/(2r)\le | \xi| \le 1/r} \text{Re} \, q(x, \xi ) \ge r^{-2} \lVert {\mathbf a}(x) \rVert$. Moreover, using Tonelli's theorem and the inequality  $1-\cos y \ge y^2/4$ for $|y| \le 1$, we  get that for all $x \in {\mathbb R}^d$ and $r>0$, 
\begin{align*}
	&\sup_{1/(2r)\le | \xi| \le 1/r} \text{Re} \, q(x, \xi ) \ge c_1 \int_{1/2 \le  |\rho|\le 1} \text{Re} \, q(x, \rho/r) d \rho \\
	&\ge c_1 \int_{{\mathbb R}^d} \int_{1/2 \le  |\rho|\le 1} (1- \cos \la z, \rho/r \ra) d \rho \, \nu(x,dz)\nn\\
	& \ge c_1 \left(  \int_{|z| \le r}  \frac{|z|^2}{r^2}  \int_{1/2 \le |\rho| \le 1} 
	\hskip-.1in
	\frac{\la z/|z|, \rho \ra^2}{4} d\rho  \nu(x, dz)   +\int_{|z|>r } \int_{ 1/2 \le |\rho | \le 1}	\hskip-.2in (1- \cos \la z/r, \rho \ra ) d \rho \, \nu(x,dz) \right)\\
	&\ge  c_1 \int_{{\mathbb R}^d} \left(\frac{ |z |^2}{r^2} \land 1 \right) \nu(x, dz)  \inf_{y \in {\mathbb R}^d, \, |y|=1} \left(  \int_{ 1/2 \le |\rho| \le 1} 	\hskip-.1in\frac{\la y, \rho \ra^2}{4} d\rho  \wedge \int_{1/2 \le |\rho| \le 1} (1- \cos \la y, \rho \ra) d\rho  \right).
\end{align*}
Let $\mathbf e_1:=(1,0,...,0)\in {\mathbb R}^d$. By the symmetry, we see that for all $y \in {\mathbb R}^d$ with $|y|=1$,
\begin{align*}
	 \int_{ 1/2 \le |\rho| \le 1}\la y, \rho \ra^2 d\rho\ge   \int_{ 1/2 \le |\rho| \le 1, \, \la \mathbf e_1, \rho \ra \ge |\rho|/2} \la \mathbf e_1, \rho \ra^2 d\rho \ge \frac{1}{16}\int_{ 1/2 \le |\rho| \le 1, \, \la \mathbf e_1, \rho \ra \ge |\rho|/2}  d\rho=c_2
\end{align*}
and
\begin{align*}
	&\int_{1/2 \le |\rho| \le 1} (1- \cos \la y, \rho \ra) d\rho =	\inf_{a>1} \int_{1/2 \le |\rho| \le 1} (1- \cos a\la \mathbf e_1, \rho \ra) d\rho \\
	&\ge 2\inf_{a>1} \int_{0}^{1/2}  \int_{\wt \rho \in {\mathbb R}^{d-1}, |\wt \rho| \le 1/2}  (1- \cos a\rho_1 )   d\wt\rho d\rho_1 \ge c_3\inf_{a>1} \int_{0}^{1/2}    (1- \cos a\rho_1 )   d\rho_1.
\end{align*}
Since $\lim_{a \to \infty} \int_{0}^{1/2}    (1- \cos a\rho_1 )   d\rho_1= 1/2$, it holds that $\inf_{a>1}\int_{0}^{1/2}    (1- \cos a\rho_1 )   d\rho_1>0$. Therefore, using the inequality $a \vee b \ge (a+b)/2$ for  $a,b \in {\mathbb R}$, we get  that  
$\sup_{1/(2r)\le  | \xi| \le 1/r}$ $\text{Re} \, q(x, \xi ) $ $\ge$ $c_4 h(x,r)$ and finish the proof.  
\end{proof}

  Recall that  $\Phi$ is defined  in \eqref{e:def-Phi}.
It is clear that $s^2h(x,s) \le r^2h(x,r)$ for all $x \in{\mathbb R}^d$ and $0<s\le r$. Thus, by \eqref{e:Phi=h}, there exists a constant $C_U'>0$ depends only on $d$ such that
\begin{equation}\label{e:Phi-upper}
	\frac{\Phi(x,r)}{\Phi(x,s)} \le C_U' \bigg(\frac{r}{s}\bigg)^{2} \quad \text{for all} \; x \in {\mathbb R}^d, \; 0<s\le r.
\end{equation}

\begin{lem}\label{l:Phi-2}
Suppose that \eqref{O3} holds for an open subset $U \subset {\mathbb R}^d$. Then there exists a constant $C_{13} \in (0,1)$ such that
$$
\inf_{|x-y|\le 2r} \Phi(y,r) \ge C_{13} \sup_{|x-y| \le 2r} \Phi(y,r) \quad \text{for all} \;\; x \in U, \; 0<4r<R_0 \wedge (C_8 \updelta_U(x)).
$$
\end{lem}
\begin{proof} By \eqref{e:Phi-upper}, Lemma \ref{l:Phi-1} and  \eqref{O3}, we get that for all $x \in U$ and $0<4r<R_0 \wedge (C_8 \updelta_U(x))$,
\begin{align*}
&\inf_{|x-y|\le 2r} \Phi(y,r) 
 \ge c_1	\inf_{|x-y| \le 2r} \bigg(	\sup_{1/(4r) \le |\xi|\le 1/(2r)} \text{Re}\, q(y,\xi) \bigg)^{-1}\\
	&= c_1	\inf_{1/(4r) \le |\xi|\le 1/(2r)}\,	\inf_{|x-y| \le 2r}  \text{Re}\, q(y,\xi)^{-1} \ge c_1 C_{10} 	\inf_{1/(4r) \le |\xi|\le 1/(2r)}\,	\sup_{|x-y| \le 2r}  \text{Re}\, q(y,\xi)^{-1}\\
		&\ge c_2	\sup_{|x-y| \le 2r} \,	\inf_{1/(4r)\le |\xi|\le 1/(2r)} \text{Re}\, q(y,\xi)^{-1} \ge c_2 C_{10} 	\sup_{|x-y| \le 2r}  \Phi(x,r).
\end{align*}
\end{proof}

 As an application of our Theorem \ref{inf}, we obtain the following  LIL for Feller processes.  See  \cite[Theorem 2]{KS} for a 1-dimensional result under  similar assumptions.

\begin{thm}\label{t:Feller1}
	Let $X$ be a Feller process on ${\mathbb R}^d$ with symbol $q$. Suppose that \eqref{O1}-\eqref{O4} hold for an open subset $U \subset {\mathbb R}^d$.  Then, there are constants $a_2 \ge a_1>0$ such that for all $x \in U$, there exists  a constant $a_x \in [a_1,a_2]$ satisfying
	\begin{equation}\label{eq:Feller1}
		\liminf_{t\to0} \frac{\Phi (x, \,\sup_{0<s\le t} | X_s - x | )}{t/\log|\log t|}=a_x,\qquad
		{\mathbb P}^x\mbox{-a.s.}
	\end{equation}
Moreover, if there exist constants $\beta_1',C_L'>0$ such that
\begin{equation}\label{e:Phi-lower}
	\frac{\Phi(x,r)}{\Phi(x,s)} \ge C_L' \bigg(\frac{r}{s}\bigg)^{\beta_1'} \quad \text{for all} \; x \in U, \; 0<s\le r<R_0 \wedge (C_8\updelta_U(x)),
\end{equation}
then there are constants $\wt a_2 \ge \wt a_1>0$  such that for all $x \in U$, there exists  a constant $\wt a_x \in [\wt a_1,\wt a_2]$ satisfying
\begin{equation}\label{eq:Feller2}
	\liminf_{t\to0} \frac{\sup_{0<s\le t} d(x, X_s)}{\Phi^{-1}(x,t/\log|\log t|)}=\wt a_x,\qquad
	{\mathbb P}^x\mbox{-a.s.}
\end{equation}
\end{thm}

\begin{remark}\label{r:a2}
	{\rm  Let $\wt \Phi$ be any function on $U \times (0,1)$ such that $\wt \Phi(x,r) \asymp \Phi(x,r)$ for $x \in U$ and $r \in (0,1)$. Thanks to the Blumenthal's zero-one law, the liminf LILs \eqref{eq:Feller1} and \eqref{eq:Feller2} hold true with $\wt \Phi$ instead of $\Phi$. Cf. Remark \ref{r:a}. }
\end{remark}

To prove  Theorem \ref{t:Feller1}, we  need the following  two lemmas.

The first one is a consequence of \cite[Section 5]{BSW}. Since we only put assumptions on the symbol $q$ locally, we carefully check ranges of variables in the proof of the following lemma.

\begin{lem}\label{l:Feller0}
	Suppose that \eqref{O1}, \eqref{O2} and \eqref{O3} hold for an open subset $U \subset {\mathbb R}^d$. Then there exist constants $C_{14}, C_{15}>0$ and $C_{16}>1$ such that for all $x \in U$, $0< \frac{8C_{12}}{C_9} r < R_0 \wedge (C_8 \updelta_U(x))$, $w \in B(x,r)$ and $t>0$,
	\begin{equation}\label{e:Pruitt1}
		{\mathbb P}^w(\tau_{B(w,r)} \le t )  \le  \frac{C_{14}t}{\Phi(x,r)},
	\end{equation}
	\begin{equation}\label{e:Pruitt2}
		{\mathbb P}^w(\tau_{B(w,r)} \ge t ) \le \exp \bigg( -  \frac{C_{15}t}{\Phi(x,r)} +1\bigg),
	\end{equation}
and
\begin{equation}\label{e:Pruitt3}
	C_{16}^{-1} \Phi(x,r) \le 	\E^w[\tau_{B(w,r)}] \le C_{16} \Phi(x,r),
\end{equation}
where $C_9,C_{10}$ and $C_{12}$ are constants in \eqref{O2}, \eqref{O3} and Lemma \ref{l:Phi-1} respectively.
\end{lem}
\begin{proof} Fix $x\in U$. Let $r_0:=R_0 \wedge (C_8 \updelta_U(x))$ and  $k:=2C_{12}/C_9>2$. Note that for all $0<4kr<r_0$ and $w \in B(x,2r)$, by the triangle inequality,
\begin{equation*}
C_8\updelta_U(w) \ge C_8(\updelta_U(x) - r) >r_0 - 2r >3kr.
\end{equation*}
Hence, by Lemma \ref{l:Phi-1} and \eqref{O2}, it holds that for all $0<4kr<r_0$ and $w \in B(x,r)$,
\begin{equation}\label{e:estimate-range}
	\sup_{1/(4kr)\le |\xi| \le 1/(2kr)} \sup_{|y-w| \le r} \frac{\text{Re} \, q(y,\xi)}{|\xi||\text{Im}\,q(y,\xi)|} \ge \frac{2kr}{C_{12}} \frac{	\sup_{|\xi| \le 1/(kr)} \text{Re} \, q(w,\xi)}{	\sup_{|\xi| \le 1/(kr)} |\text{Im}\,q(w,\xi)|} \ge  \frac{2C_9kr}{C_{12}} = 4r.
\end{equation}

 Using \eqref{O3}, Lemma \ref{l:Phi-1} and the monotone property of $\Phi$, we get that for all $0<4kr < r_0$ and  $w \in B(x,r)$,
\begin{align}\label{e:infsup_q}
	&\sup_{1/(4kr)\le |\xi| \le 1/(2kr)} \,	\inf_{|y-w| \le 3r} \text{Re} \, q(y, \xi) \nn\\
	&\ge 	\sup_{1/(4kr)\le |\xi| \le 1/(2kr)} \, 	\inf_{|y-x| < 4r} \text{Re} \, q(y, \xi) \ge C_{10}  \sup_{1/(4kr)\le |\xi| \le 1/(2kr)}  \text{Re} \, q(x, \xi) \ge  \frac{C_{10}}{C_{12}} \Phi(x,r).  
\end{align}

On the other hand, by Lemma \ref{l:Phi-1} and \eqref{O3}, we also get that for all $0<4r < r_0$ and  $w \in B(x,r)$,
\begin{equation}
	\sup_{|y-w|\le r} \, \sup_{|\xi| \le 1/r} \text{Re}\,q(y, \xi) \le C_{12} \sup_{1/(2r) \le |\xi| \le 1/r} \, \sup_{|y-w|\le r} \text{Re}\,q(y, \xi)\le \frac{C_{12}}{C_{10}} \Phi(x,r)^{-1}.
\end{equation}
Moreover, we get from Lemma \ref{l:Phi-2} and \eqref{O2} that for all $0<4kr < r_0$ and  $w \in B(x,r)$,
\begin{align*}
 \sup_{|y-w| \le r} \,\sup_{4/r_0<|\xi|\le 1/r} |\text{Im}\, q(y,\xi)|  \le \frac{1}{C_9}  \sup_{|y-x|<2r}  \,\sup_{|\xi|\le 1/r} |\text{Re}\, q(y,\xi)| \le \frac{C_{13}}{C_9} \Phi(x,r)^{-1}.
\end{align*}
Using the triangle inequality several times, \eqref{O2} in the second inequality, the inequality $|a-\sin a|\le |a|^2$ for $a \in {\mathbb R}$ in the third, Lemma \ref{l:Phi-1} in the fourth, and the monotone property of $\Phi$ and Lemma \ref{l:Phi-2} in the last, we get that for all $y,\xi\in {\mathbb R}^d$ such that $|y-x| < 2r$ and $|\xi| \le 4/r_0$, 
\begin{align*}
	|\,\text{Im} \, q(y,\xi)|  & \le \frac{r_0|\xi|}{4} \big|\,\text{Im} \, q(y,\frac{4}{r_0|\xi|}\xi)\big| + \bigg| \int_{{\mathbb R}^d } \bigg( \frac{r_0|\xi|}{4} \sin \la z, \frac{4}{r_0|\xi|}\xi\ra- \sin \la z, \xi \ra\bigg)\nu(y, dz) \bigg| \\
	& \le  \frac{r_0|\xi|}{4C_9} \Phi(y,4/r_0)^{-1}+ \frac{r_0|\xi|}{4} \int_{|z| < r_0/4 }   \bigg| \la z, \frac{4}{r_0|\xi|}\xi \ra -   \sin \la z, \frac{4}{r_0|\xi|}\xi\ra \bigg| \,\nu(y, dz)    \\
	& \quad   +   \int_{|z| < r_0/4}  \big|\la z, \xi \ra- \sin \la z, \xi \ra \big|\nu(y, dz)   + \bigg( \frac{r_0|\xi|}{4} + 1 \bigg)  \int_{|z| \ge r_0/4} \nu(y, dz)  \\
		& \le  \frac{1}{C_9} \Phi(y,4/r_0)^{-1}+ \bigg(\frac{16}{r_0^2} + \frac{16}{r_0^2}\bigg) \int_{|z| < r_0/4 }  |z|^2\nu(y, dz)   + 2 \nu\big(y,{\mathbb R}^d \setminus B(0, r_0/4) \big)  \\
		& \le c_1 \Phi(y,4/r_0)^{-1} \le c_2 \Phi(x,r)^{-1}.
\end{align*}
Therefore, we deduce that for all $0<4kr<r_0$ and $w \in B(x,r)$,
\begin{align}\label{e:supsup_q}
	\sup_{|y-w|\le r} \sup_{|\xi| \le 1/r} |q(y, \xi)| \le  	\sup_{|y-w|\le r} \Big( \sup_{|\xi| \le 1/r} \text{Re}&q(y, \xi)
	+	 \sup_{4/r_0<|\xi| \le 1/r} |\text{Im}\,q(y, \xi)|\nn\\
	& +  \sup_{|\xi| \le 4/r_0} |\text{Im}\,q(y, \xi)| \Big)\le  c_3\Phi(x,r)^{-1}.
\end{align}

Finally, by \eqref{e:estimate-range}, \eqref{e:infsup_q} and  \eqref{e:supsup_q}, we obtain the results  from \cite[Theorem 5.1, Corollary 5.3 and Theorem 5.9]{BSW}. \end{proof}

\begin{lem}\label{l:Feller1}
	Suppose that \eqref{O1}-\eqref{O4} hold for an open subset $U \subset {\mathbb R}^d$. Then  \eqref{A4} holds  for $U$.
\end{lem}
\begin{proof}  By following the proof of \cite[Proposition 5.2]{GRT19}, one can  deduce \eqref{A4}  from \eqref{e:Pruitt1},  \eqref{e:Pruitt2} and \eqref{O4}. Below, we give  the proof of this conclusion  in details for the reader's convenience.

Choose any $x \in U$ and set $r_0:= R_0 \wedge (C_8 \delta_U(x))$. For all $w \in B(x, r_0/8)$, since $\updelta_U(w) \ge \updelta_U(x)-r_0/8>\updelta_U(x)/5$, we get from Lemma \ref{l:Phi-2} and \eqref{e:Phi-upper} that 
\begin{equation}\label{e:Feller1-1}
	\Phi\big(w, R_0 \wedge (C_8\updelta_U(w))\big) \ge \Phi(w, r_0/5 ) \ge  C_{13} \Phi(x, r_0/5 ) \ge c_1	\Phi(x,r_0),
\end{equation}
where the constant $c_1>0$ is independent of $x$.

Let $0<\frac{8C_{12}}{C_9}r <r_0$. 
By \eqref{e:Pruitt1} and \eqref{e:Pruitt3}, there is a constant $\eps_1 \in (0,c_1\eps_0/C_{16})$ independent of $x$ and $r$ such that for  all $w \in B(x,r/2)$,
\begin{equation}\label{e:def-unittime}
{\mathbb P}^w(\tau_{B(w,r/2)} \le \eps_1 \E^x[\tau_{B(x,r)}] )  \le \frac{C_{11}}{2},
\end{equation}
where $\eps_0$, $C_{11}$ and  $C_{16}$ are the constants in \eqref{O4} and \eqref{e:Pruitt3}. We set $t_0:= \eps_1 \E^x[\tau_{B(x,r)}]$ and for $k \ge 0$,
$$
S_k:= \left\{ \sup_{kt_0 \le u \le (k+1)t_0}|X_u-X_{kt_0}|<\frac{r}{2}, \, 2\la X_{(k+1)t_0}-X_{kt_0}, X_{kt_0} \ra \le -|X_{(k+1)t_0}-X_{kt_0}| |X_{kt_0}| \right\}.
$$
Note that, for any $y,z \in {\mathbb R}^d$, if $|y|, |z| < r/2$ and $2\la y,z \ra \le -|y||z|$, then  $|y+z| < r/2$. Thus, by the Markov property, we see that for all $n \ge 1$, 
\begin{align}\label{e:SP_low}
	{\mathbb P}^{x}(\tau_{B(x,r)} \ge n t_0) &\ge {\mathbb P}^{x}\Big( \cap_{k=0}^{n-1} S_k\Big) = {\mathbb P}^{x}\E\Big[ \cap_{k=0}^{n-1} S_k \, \big| \, \FF_{(n-1)t_0} \Big] \nn\\
	& \ge \inf_{w \in B(x,r/2)} {\mathbb P}^w(S_0) \cdot {\mathbb P}^{x}\Big( \cap_{k=0}^{n-2} S_k\Big) \ge ... \ge \Big( \inf_{w \in B(x,r/2)} {\mathbb P}^w(S_0) \Big)^n. \qquad
\end{align}
For all $w \in B(x,r/2)$, since $\eps_0 	\Phi\big(w, R_0 \wedge (C_8\updelta_U(w))\big) > c_1 \eps_0 C_{16}^{-1}\E^x[\tau_{B(x,r)}]>t_0$ by \eqref{e:Feller1-1} and \eqref{e:Pruitt3}, using
\eqref{O4} and \eqref{e:def-unittime}, we get that
\begin{align*}
	{\mathbb P}^w(S_0) &\ge 1 -{\mathbb P}^w(2\la X_{t_0}-w, w \ra > -|X_{t_0}-w| |w|)-{\mathbb P}^w(\tau_{B(w, r/2)} \le t_0) \ge \frac{C_{11}}{2}.
\end{align*}
It follows that for all $n \ge 1$,
$$
{\mathbb P}^x(\tau_{B(x,r)} \ge n \E^x[\tau_{B(x,r)}])  = {\mathbb P}^x(\tau_{B(x,r)} \ge n\eps_1^{-1} t_0) \ge \bigg( \frac{C_{11}}{2} \bigg)^{-\eps_1^{-1}n +1}.
$$
On the other hand, by \eqref{e:Pruitt2} and \eqref{e:Pruitt3}, we get that for all $n \ge 1$,
$$
{\mathbb P}^x(\tau_{B(x,r)} \ge n \E^x[\tau_{B(x,r)}]) \le e^{-\frac{C_{15}}{C_{16}} n+1 }.
$$
The proof is complete.  \end{proof}

\noindent \textbf{Proof of Theorem \ref{t:Feller1}.} With a redefined $R_0$, in view of \eqref{e:Pruitt3}, we get \eqref{A1} from Lemma \ref{l:Phi-2}, \eqref{A2} from  \eqref{O1} and \eqref{e:Phi-upper}, \eqref{A3} from Lemma \ref{l:Phi-1}, and \eqref{A4} from Lemma \ref{l:Feller1}. Then by the proof of Theorem \ref{inf}, one can see that  \eqref{e:inf3} below holds for all $x \in U$ with $\Phi(x, \cdot)$ instead of $\phi(x, \cdot)$. Hence, by the Blumenthal's zero-one law, we deduce \eqref{eq:Feller1}. Moreover, if \eqref{e:Phi-lower} also holds true, then we arrive at \eqref{eq:Feller2} by a similar argument to that in the proof of  Corollary \ref{c:inf}. We omit  details here. \qed 

We now give concrete examples of Feller processes which satisfy both   liminf LIL at zero and limsup LIL at zero with help from the paper \cite{CKL}. In the following two examples, we use conditions Tail, E and NDL introduced in \cite{CKL}. See Definitions \ref{d:0}--\ref{d:inf} in  Appendix for their definitions.

\begin{example}\label{E:non}{\bf (Non-symmetric Feller processes)}
	{\rm
		Let $\nu$ be a  nonincreasing nonnegative function on $(0,\infty)$ satisfying $\int_0^\infty (r^{d-1} \land r^{d+1}) \nu(r)dr < \infty$. 
			Define for $r>0$,
		$$\mathcal{G}(r)= 1\bigg/\int_{|z| \ge r} \nu(|z|) dz  \quad \text{and} \quad  \mathcal{H}(r)=1\bigg/ \int_{{\mathbb R}^d } \left( \frac{|z|^2}{r^2} \land 1 \right) \nu(|z|) dz.$$
		Then $\mathcal H \le \mathcal G$, $\mathcal H$ is increasing and $\sH(r)/\sH(s) \le (r/s)^2$ for all $0<s\le r$. We assume that there exist constants $0<\beta_1\le\beta_2 \le 2$ and $c_1,c_2>0$ such that 
		\begin{equation}\label{e:sH-scaling}
		c_1\bigg(\frac{r}{s}\bigg)^{\beta_1}\le 	\frac{\sH(r)}{\sH(s)} \le c_2\bigg(\frac{r}{s}\bigg)^{\beta_2} \quad \text{for all} \;\, 0<s \le r \le 1.
		\end{equation}
	Let $J(z)$ be a nonnegative function on ${\mathbb R}^d$ comparable to $\nu(|z|)$ and  $\kappa(x,z)$  be a Borel function on ${\mathbb R}^d \times {\mathbb R}^d$ such that for some constants $a_1,a_2,a_3>0$ and $\beta \in (0,1)$,
	\begin{equation}\label{e:non-kappa}
		a_1 \le \kappa(x,z) \le a_2 \quad  \text{and} \quad |\kappa(x,z) - \kappa(y,z)| \le a_3 |x-y|^\beta \quad \text{for all} \;\, x,y,z \in {\mathbb R}^d.		
	\end{equation}
In this example, we always 	suppose that  one of the following assumptions holds true:
	
	\smallskip
	
	(P1) \eqref{e:sH-scaling} holds with $\beta_1>1$,
	
	(P2) \eqref{e:sH-scaling} holds with $\beta_2<1$,
	
	(P3) $J(z)=J(-z)$ and $\kappa(x,z)=\kappa(x,-z)$ for all $x,z \in {\mathbb R}^d$.
	
	\medskip
	
\noindent	In each case when (P1), (P2) and (P3) holds, respectively, we consider an operator		
\begin{align*}
	&\sL^\kappa f(x) := 
	\begin{cases}
	\int_{{\mathbb R}^d}  \big( f(x+z) - f(x)  - \1_{|z| \le 1} \langle z, \nabla f(x) \rangle \big) \kappa(x,z) J(|z|) dz,  & \text{if (P1) holds}; \\
 \int_{{\mathbb R}^d}  \big( f(x+z) - f(x) \big) \kappa(x,z) J(|z|)dz,  & \text{if (P2) holds}; \\
\frac{1}{2} \int_{{\mathbb R}^d} \big( f(x+z) + f(x-z) - 2f(x) \big) \kappa(x,z)J(|z|)dz & \text{if (P3) holds}. 
	\end{cases}
					\end{align*}
	According to \cite[Theorem 1.3 and Remark 1.5]{GS19}, if \eqref{e:non-kappa} and one among (P1)-(P3) hold, then 
		there exists a Feller process $X$ on ${\mathbb R}^d$ whose infinitesimal generator is an extension of $(\sL^\kappa, C_c^2({\mathbb R}^d))$. Indeed, the process $X$ is the unique solution to the martingale problem for $(\sL^\kappa, C_c^\infty({\mathbb R}^d))$.	
By Lemma \ref{l:Phi-1} and \eqref{e:non-kappa}, since $J(z)$ is comparable to $\nu(|z|)$,  the symbol $q$ of $X$ satisfies that 
 \begin{equation}\label{e:sH-Phi}
 \sup_{| \xi| \le 1/r} \text{\rm Re} \, q(x, \xi )  \asymp 1/\sH(r)  \quad \text{for} \;\, x \in {\mathbb R}^d, \; r>0
 \end{equation} 

In the followings, we check that $X$ satisfies conditions \eqref{O1}-\eqref{O4} for $U={\mathbb R}^d$. 

First, we note that, by \eqref{e:sH-scaling} and \eqref{e:non-kappa},  $\int_{{\mathbb R}^d} \kappa(x,z) J(z)dz \ge c\int_{{\mathbb R}^d}\nu(|z|)dz= $ $c\lim_{r \to 0} 1/\sH(r)$ $=\infty$  for all $x \in {\mathbb R}^d$. Hence, \eqref{O1} holds  for $U={\mathbb R}^d$.

			When (P3) holds, the symbol $q(x,\xi)$ is a real number for all $x, \xi$ so that  \eqref{O2} for $U={\mathbb R}^d$ immediately follows. 
		We now check \eqref{O2} for  the cases (P2) and (P3) separately. 	
		
Suppose  (P1) holds.  Using the triangle inequality,  Taylor expansion for the sine function and \eqref{e:sH-scaling}, since $\kappa$ is bounded above, $J(z)$ is comparable to $\nu(|z|)$. Since  $\nu(|z|)dz$ is a L\'evy measure on ${\mathbb R}^d$ and $\beta_1>1$, we get that for all $x\in {\mathbb R}^d$ and $\xi \in {\mathbb R}^d$ with $|\xi|\ge 1$,
\begin{align*}
	|\text{Im}\, q(x,\xi) | 	&= \bigg|  \int_{{\mathbb R}^d \setminus\{0\}}  \big( \langle z,\xi \rangle \1_{\{|z| \le 1\}}- \sin \langle z, \xi \rangle \big)  \kappa(x,z)J(z)dz \bigg| \\
	&\le \bigg|  \int_{|z| \le  1} \big( \langle z,\xi \rangle - \sin \langle z, \xi \rangle \big)  \kappa(x,z)J(z)dz \bigg| + \bigg|  \int_{|z| > 1}  \sin \langle z, \xi \rangle   \kappa(x,z)J(z)dz \bigg|\\&\le c_2 \int_{|z| \le  1} \big( (|z||\xi|)^3 \wedge (|z||\xi|)  \big)  \nu(|z|)dz  + c_2   \int_{|z| > 1} \nu(|z|)dz \\
	&\le c_2  \int_{|z| \le 1/|\xi|} (|z||\xi|)^2 \nu(|z|)dz + c_2  |\xi| \int_{1/|\xi|<|z| \le 1} |z|^2 \nu(|z|)dz +c_3 \\
	&\le \frac{c_4}{\sH(1/|\xi|)} + \frac{c_4|\xi|+ c_3\sH(1)}{\sH(1)}  	\le \frac{c_4}{\sH(1/|\xi|)} + \frac{(c_4+ c_3\sH(1)) |\xi|}{c_1 |\xi|^{\beta_1}\sH(1/|\xi|)} \le \frac{c_5}{\sH(1/|\xi|)}.
\end{align*}
Suppose (P2) holds. Using \eqref{e:sH-scaling}, since $\kappa$ is bounded above, $J(z)$ is comparable to $\nu(|z|)$. Using this fact  and $\beta_2<1$,  we see that for all $x \in {\mathbb R}^d$ and $\xi \in {\mathbb R}^d$ with $|\xi|\ge 1$,
\begin{align*}
	&|\text{Im}\, q(x,\xi) | =\bigg|  \int_{{\mathbb R}^d \setminus \{0\} }  \sin \langle z, \xi \rangle  \kappa(x,z)J(z)dz\bigg|\\
	&\le c_6|\xi| \int_{|z|<1/|\xi|}  |z|  \nu(|z|)dz + c_6\int_{|z| \ge 1/|\xi|} \nu(|z|)dz\\
&\le c_6   \sum_{n \ge 1} \int_{2^{-n}/|\xi|\le |z|<2^{-n+1}/|\xi|}  2^{-n+1}  \nu(|z|)dz + \frac{c_6}{\sH(1/|\xi|)}\\
&\le   c_6\sum_{n \ge 1}  \frac{ 2^{-n+1}}{\sH(2^{-n}/|\xi|)} + \frac{c_6}{\sH(1/|\xi|)} \le    \frac{c_2c_6}{\sH(1/|\xi|)} \sum_{n \ge 1}  2^{-n+1+n\beta_2} + \frac{c_6}{\sH(1/|\xi|)}=\frac{c_7}{\sH(1/|\xi|)}.
\end{align*}
Therefore, by \eqref{e:sH-Phi}, we deduce that \eqref{O2} always holds true for $U={\mathbb R}^d$.	

 \eqref{O3}	 
		 immediately
		  follows from  the fact that $\kappa(x,z)$ is bounded above and below by positive constants.
		For \eqref{O4}, we see from  \cite[(84)]{GS19} that the heat kernel $p(t,x,y)$ of $X$  satisfies that for all  $t \le 1$ and $x,y \in {\mathbb R}^d$ with $|x-y| \le \sH^{-1}(t)$,
		$$ p(t,x,y) \ge c_8 \sH^{-1}(t)^{-d}. $$
It follows that for all $t \le 1$, $x \in {\mathbb R}^d$ and $z \in {\mathbb R}^d$ with $|z|=1$,
		\begin{align*}
		&{\mathbb P}^x\big(2\la  X_t- x, z \ra \le -| X_t -x |  \big) \ge \int_{B(x,\sH^{-1}(t)) \cap \{ y : 2\la  y- x, z \ra \le -| y -x |  \}} p(t,x,y) dy \nn\\
		& \ge c_8 \sH^{-1}(t)^{-d} \int_{B(x,\sH^{-1}(t)) \cap \{ y : 2\la  y- x, z \ra \le -| y -x |  \}} dy =  c_8 \int_{B(0,1) \cap \{ y : 2\la  y, z \ra \le -| y|  \}} dy =c_9.
		\end{align*}
		Therefore, \eqref{O4} holds true for $U={\mathbb R}^d$. 
		
	Now, using  \eqref{e:sH-scaling} and \eqref{e:sH-Phi}, we conclude from  Theorem \ref{t:Feller1} and Remark \ref{r:a2} that the liminf LIL at zero \eqref{eq:Feller2} holds true with $U={\mathbb R}^d$ and $\Phi^{-1}(x, t/\log|\log t|)$ replaced by $\sH^{-1}(t/\log|\log t|)$.

		\medskip

			To obtain a limsup LIL at zero for the Feller process $X$,  we also assume that  $\U_{r_1}(\mathcal G, \gamma, c)$
			(in Definition \ref{d:ws})
			 holds for some $\gamma>0$ and $r_1, c\in (0,1)$. Here, we emphasize that $\gamma$ may not be smaller than $2$.	 By \eqref{e:non-kappa}, 
		$\mathrm{Tail}_\infty(\mathcal G,{\mathbb R}^d)$ and  $\mathrm{Tail}_\infty(\sH,{\mathbb R}^d, \le)$
				(in Definition  \ref{d:inf}(i))
		 hold. Moreover, by \cite[Theorem 1.2(4) and Lemma 4.11]{GS19} and our Lemma \ref{l:hkendl},
		$\mathrm{NDL}_{R_0'}(\sH,{\mathbb R}^d)$ (in Definition  \ref{d:inf}(iii)) holds for some $R_0'>0$. 
	 Therefore, using \eqref{e:sH-scaling}, we deduce that  for all $x \in {\mathbb R}^d$, the limsup LIL at zero given in \cite[Theorem 1.11(i-ii)]{CKL} holds true for $X$ with functions $\phi=\sH$ and $\psi=\mathcal G$.
		\qed
	}
\end{example}

 In the following example, we directly check that conditions \eqref{A1}-\eqref{A4}, \eqref{B1}-\eqref{B3} and \eqref{B4+}  hold.
\begin{example}\label{E:KKK19}{\bf (Singular L\'evy measure)}
	{\rm  
	Let $\alpha \in (0,2)$, $d \ge 2$ and ${\mathbb R}_i := \{(x_1,\dots,x_d) \in {\mathbb R}^d: x_j = 0 \mbox{ if } j \neq i  \}$ for $1\le i \le d$. Denote by $e^i$, $1\le i \le d$ the standard unit vectors in ${\mathbb R}^d$.  Define a kernel $J(x,y)$ on ${\mathbb R}^d \times {\mathbb R}^d$ by
		\begin{equation}\label{e:axis}
			J(x,y) = \begin{cases}
				b(x,y)|x-y|^{-1-\alpha}, \;\; &\mbox{if} \;\; y-x \in \cup_{i=1}^d {\mathbb R}_i \setminus \{0\},\\
				0, \;\; &\mbox{otherwise},
			\end{cases}
		\end{equation}
		where  $b(x,y)$ is a symmetric function on ${\mathbb R}^d \times {\mathbb R}^d$ that is bounded between two positive constants. Using this kernel, define a symmetric form $(\EE, \FF)$ on $L^2({\mathbb R}^d;dx)$ as
		\begin{equation*}
			\begin{split}
				\EE(u,v) &= \int_{{\mathbb R}^d} \Big( \sum_{i=1}^d \int_{{\mathbb R}} (u(x+e^i \tau) - u(x) )(v(x+e^i \tau) - v(x)) J(x,x+e^i \tau)d\tau   \Big)dx,\\
				\FF &= \{ u \in L^2({\mathbb R}^d; dx) \, | \,  \EE(u,u) < \infty  \}.
			\end{split}
		\end{equation*}
		According to \cite[Theorem 3.9 and Corollary 4.15]{Xu13}, the above form $(\EE, \FF)$ is a regular Dirichlet form and the associated Hunt process $X$ is a strong Feller process in ${\mathbb R}^d$. Below, we get liminf and limsup LILs for $X$, both at zero and at infinity.
		
	From the definition \eqref{e:axis}, we see that  $\mathrm{Tail}_\infty(r^\alpha, {\mathbb R}^d)$ holds true. Indeed, for all $x \in {\mathbb R}^d$ and $r>0$, $\int_{B(x,r)^c} J(x,dy) = \sum_{i=1}^d \int_{|\tau| \ge r} J(x,x+\tau e^i) d\tau \asymp d\int_{|\tau| \ge r} \tau^{-1-\alpha}d\tau =c r^{-\alpha}$. 
		By \cite[Proposition 4.4 and the proof of Theorem 4.6]{Xu13}, there exist  $c_1,c_2>0$ such that for all $x \in {\mathbb R}^d$, $r,t>0$ and $n \in \N$, 
		 \begin{equation}\label{e:KKK191}
		{\mathbb P}^x(\tau_{B(x,r)} <t) \le c_1tr^{-\alpha}  \quad \text{and} \quad  \;	{\mathbb P}^x(\tau_{B(x,r)} \ge c_2nr^\alpha) \le 2^{-n}.
		\end{equation}
		By \cite[Proposition 4.18]{Xu13}, there exist $c_3,c_4>0$ such that
		\begin{equation}\label{e:LHK-singular}
			p(t,x,y) \ge c_3 t^{-d/\alpha} \quad \text{for all} \;\, t>0, \; x,y \in {\mathbb R}^d \text{ with } |x-y| \le c_4t^{1/\alpha}.
		\end{equation} 
It follows that for all $x \in {\mathbb R}^d \setminus \{0\}$ and $t>0$,
		\begin{equation}\label{e:KKK192}
			{\mathbb P}^x\big(2\la X_t-x, x \ra \le - |X_t-x||x| \big) \ge c_3t^{-d/\alpha} \int_{2\la y-x, x \ra \le - |y-x||x|, \, |y-x| \le c_4t^{1/\alpha}} dy \ge c_3 c_5(d),
		\end{equation}
		for  a constant $c_5(d)>0$ which only depends on the dimension $d$.

		Now, by using the first inequality in  \eqref{e:KKK191} and \eqref{e:KKK192}, one can repeat the proof of Lemma \ref{l:Feller1} and deduce that for all $x \in {\mathbb R}^d$, $r>0$ and $n \in \N$,
		 \begin{equation}\label{e:KKK191-2}
		{\mathbb P}^x(\tau_{B(x,r)} \ge c_6nr^\alpha) \ge c_7e^{-c_8n}
		\end{equation}
	with some  $c_6,c_7,c_8>0$. From the latter inequality in \eqref{e:KKK191} and \eqref{e:KKK191-2}, we get that $\E^x[\tau_{B(x,r)}]$ $\asymp r^\alpha$ for $x \in {\mathbb R}^d$ and $r>0$, and conditions \eqref{A4} with $U={\mathbb R}^d$ and \eqref{B4} hold true. Consequently, all conditions \eqref{A1}-\eqref{A3} (with $U={\mathbb R}^d$) and \eqref{B1}-\eqref{B3} are satisfied since we already checked that $\mathrm{Tail}_\infty(r^\alpha, {\mathbb R}^d)$  holds true. Moreover, using H\"older continuity of the heat kernel given in \cite[Corollary 4.19]{Xu13} and \eqref{e:LHK-singular}, one can repeat the proof of \cite[Proposition 4.15]{CKL} and deduce that the zero-one law for shift-invariant events stated in Proposition \ref{p:law01} holds true.
	
	Eventually, from Corollaries \ref{c:inf} and \ref{c:infinf}, and Remark \ref{r:a},  
	we conclude that for all $x,y \in {\mathbb R}^d$, both liminf LILs \eqref{c1} and \eqref{2} hold with $\phi(x,r)=\phi(r) = r^\alpha$. Also, we  conclude from \cite[Theorems 1.11-1.12]{CKL} that the limsup LILs \cite[(1.12) and (1.15)]{CKL} hold with $\phi(r)=r^\alpha$. 
	\qed
	}
\end{example}

Using  { \textit{the local symmetrization} } which has been introduced in \cite{SW13}, we can obtain  a sufficient condition for \eqref{O4} in terms of the symbol $q(\cdot,\xi)$. We introduce the following condition:

\medskip

\setlength{\leftskip}{3mm}

\noindent {\bf (S)}  $C_c^\infty({\mathbb R}^d)$ is an operator core for  $(\sL, \sD(\sL))$, i.e.   $\overline{\sL|_{C_c^\infty({\mathbb R}^d)}}=\sL$, and there exist  constants  $R_0,  A_0\in (0,1)$, $K_0 \ge 1$ and $c_L,c_U>0$ such  that the following conditions hold for every $x \in U$:

\begin{enumerate}[(i)]
	\item There exists an increasing function $g(x,\cdot)$ and constants $0<\alpha(x)\le \beta(x)$  such that 
	\begin{equation}\label{e:alphabeta}
		\frac{1}{\alpha(x)}-\frac{1}{\beta(x)} < \frac{1}{d^2+d},
	\end{equation}
\begin{equation}\label{e:g-scaling}
c_L \bigg(\frac{r}{s} \bigg)^{\alpha(x)} \le 	\frac{g(x,r)}{g(x,s)} \le c_U \bigg(\frac{r}{s} \bigg)^{\beta(x)} \quad \text{for all} \;\; r \ge s> 1/(R_0 \wedge (A_0 \updelta_U(x))).
\end{equation}
	and
	\begin{equation}\label{e:g-comp}
		K_0^{-1}g(x,|\xi|) \le \text{Re}\, q(x, \xi) \le K_0 g(x,|\xi|)\; \text{ for all }  \xi \in {\mathbb R}^d, \; |\xi|>1/(R_0 \wedge (A_0 \updelta_U(x))).
	\end{equation}
	\item For every $0<r<R_0 \wedge (A_0 \updelta_U(x))$, there exists  a Feller process $Y=Y^{x,r}$ with symbol $q_Y(\cdot,\xi)$ such that
	\begin{align}
		&(a) \;  q(y, \xi)=2 \text{Re}\,q_Y(y,\xi/2) \; \text{ for all } y \in \overline{B(x,r)} \text{ and } \xi\in {\mathbb R}^d,\label{a:as2}\\[5pt]
		&(b) \; K_0^{-1} \inf_{|z-x| \le r} \textrm{Re}\,q_Y(z,\xi) \le  \textrm{Re}\,q_Y(y,\xi) \le K_0 \sup_{|z-x|\le r} \textrm{Re }q_Y(z,\xi) \\
		&\quad\;\; \text{for all } y \in {\mathbb R}^d \setminus \overline{B(x,r)} \text{ and } \xi\in {\mathbb R}^d, \; |\xi|> 1/(R_0 \wedge (A_0 \updelta_U(x))). \label{a:as3}
	\end{align}
\end{enumerate}

The condition {\bf (S)} looks complicated but it is quite straightforward to check when some concrete form of $q(x,\xi)$ is given. See Examples \ref{E:Feller} and \ref{E:singular} below.

\begin{remark}\label{r:well-posed}
	{\rm The assumption that $C_c^\infty({\mathbb R}^d)$ is an operator core for $(\sL, \sD(\sL))$ is equivalent to the well-posedness of the martingale problem for $(-q(\cdot, D), C_c^\infty({\mathbb R}^d))$.	See \cite[Proposition 4.6]{SW13}.
	}
\end{remark}

\begin{lem}\label{l:(S)}
	Suppose that \eqref{e:g-comp} holds. Then there exists a constant $K_1>1$ such that
	\begin{equation*}
	\frac{K_1^{-1}}{\Phi(x,r)} \le g(x,1/r) \le 	\frac{K_1}{\Phi(x,r)}\; \text{ for all } \; x\in U \; \text{and } \; 0<2r<R_0 \wedge (A_0 \updelta_U(x)).
	\end{equation*}
\end{lem}
\begin{proof} Using \eqref{e:g-comp}, Lemma \ref{l:Phi-1} and the monotonicity of $g$, we get that for all $x \in U$ and $0<2r<R_0 \wedge (A_0 \updelta_U(x))$,
\begin{equation*}
\frac{1}{\Phi(x,r)} \ge 	\sup_{|\xi|=1/r}\text{Re} \, q(x, \xi) \ge K_0^{-1} g(x,1/r)
\end{equation*}
and
\begin{equation*}
	\frac{1}{\Phi(x,r)} \le C_{12} \sup_{1/(2r) \le |\xi|\le 1/r}\text{Re} \, q(x, \xi) \le  C_{12}K_0 g(x,1/r).
\end{equation*}\end{proof}

\begin{prop}\label{p:FellerO4}
	Suppose that  \eqref{O3} and {\bf (S)}  hold. Then  \eqref{O1}, \eqref{O2} and  \eqref{O4}  hold.
\end{prop}
\begin{proof} The first inequality in \eqref{e:g-scaling} implies that $\lim_{r \to \infty} g(x,r)=\infty$ for all $x \in U$. Hence,  we get \eqref{O1} from Lemma \ref{l:(S)}. \eqref{O2} is obvious because $q(x,\xi)$ is real for all $x \in U$ and $\xi \in {\mathbb R}^d$ by \eqref{a:as2}. Thus it remains to prove \eqref{O4}.

Fix $x_0 \in U$ and set $r_0:=8^{-1} (R_0 \wedge (A_0 \updelta_U(x_0))$.   
Let $\eta_0 \in (0,1)$ be a constant which will be chosen later. Pick  any $0<t_1< \eta_0\Phi(x_0, r_0)$  and then define $r_1=\Phi^{-1}(x_0,\eta_0^{-1}t_1) \in (0, r_0)$. Let $Y=Y^{x_0,2r_1}$ be a Feller process on ${\mathbb R}^d$ satisfying \eqref{a:as2} and \eqref{a:as3} with $x=x_0$ and $r=2r_1$. Denote by  $Y'$  an independent copy of $Y$ and set  $ Y^S_t:=\frac{1}{2} (Y_t+2Y'_0-Y'_t)$. Then according to \cite[Lemma 2.8]{SW13}, $Y^S$ is a Feller process with symbol $2\text{Re}\,q_Y(\cdot, \xi/2)$ and its characteristic function $\lambda_t(y,\xi):=\E^{y}[e^{i\la Y^S_t-y, \xi\ra}]$
is nonnegative for every $t \ge 0$ and $y,\xi \in {\mathbb R}^d$. 

Since the martingale problem for $(-q(\cdot, D), C_c^\infty({\mathbb R}^d))$ is well-posed (Remark \ref{r:well-posed}),  by \cite[Theorem 5.1]{Hoh},  the stopped martingale problem for $(-q(\cdot, D), C_c^\infty({\mathbb R}^d))$ and $B(x_0, 2r_1)$  is also well-posed. Therefore, by constructing $X$ and $Y^S$ in the same probability space, we may assume that $X_s$ and $Y^S_s$ have the same distribution for $0\le s<\tau_{B(x_0,r_1)}$ under ${\mathbb P}^{x_0}$.  
Then using \eqref{e:Pruitt1}, we get that for all $z \in {\mathbb R}^d$ with $|z|=1$,
\begin{align}\label{e:O4-main}
	{\mathbb P}^{x_0}\big(2\la  X_{t_1}- x_0, z \ra \le -| X_{t_1} -x_0 |  \big) &\ge 	{\mathbb P}^{x_0}\big(2\la  Y^S_{t_1}- x_0, z \ra \le -| Y^S_{t_1} -x_0 |, \; t_1<\tau_{B(x_0,r_1)} \big) \nn\\
	& \ge 	{\mathbb P}^{x_0}\big(2\la  Y^S_{t_1}- x_0, z \ra \le -| Y^S_{t_1} -x_0 |  \big) - {\mathbb P}^{x_0} \big( \tau_{B(x_0, r_1)} \le t_1\big)\nn\\
		& \ge 	{\mathbb P}^{x_0}\big(2\la  Y^S_{t_1}- x_0, z \ra \le -| Y^S_{t_1} -x_0 |  \big) - C_{14} \eta_0.
\end{align}

 For simplicity, we denote $\alpha$ for $\alpha(x_0)$ and $\beta$ for  $\beta(x_0)$. Using \eqref{a:as3}, \eqref{e:g-comp}, \eqref{a:as2}, \eqref{e:g-scaling}, and Lemmas \ref{l:(S)} and \ref{l:Phi-2}, we get that  for all $u<2r_1$,
\begin{align}\label{e:H-W-condition}
	&\inf_{z \in {\mathbb R}^d} \inf_{|\xi|=1/u} \text{Re} \,q_Y(z, \xi) \ge \frac{1}{K_0} \inf_{z \in B(x_0, 2r_1)} \inf_{|\xi|=1/u} \text{Re} \,q_Y(z, \xi)  \ge\frac{1}{2K_0^2} \inf_{z \in B(x_0, 2r_1)}  g(z, 2/u) \nn \\
	&\ge \frac{c_L}{2K_0^2} \left( \frac{2r_1}{u}\right)^\alpha\inf_{z \in B(x_0, 2r_1)} g(z, 1/r_1)  \ge  \frac{c_L}{2K_0^2K_1} \left( \frac{2r_1}{u}\right)^\alpha\inf_{z \in B(x_0, 2r_1)}\frac{1}{\Phi(z, r_1)}\nn\\
	& \ge   \frac{c_LC_{13}}{2K_0^2K_1} \left( \frac{2r_1}{u}\right)^\alpha \frac{1}{\Phi(x_0, r_1)} = \frac{c_LC_{13}}{2K_0^2K_1}  \left( \frac{2r_1}{u}\right)^\alpha \frac{\eta_0}{t_1}.
\end{align}
In particular, we have
\begin{equation*}
	\lim_{|\xi| \to \infty} \frac{\inf_{z \in{\mathbb R}^d} \text{Re} \,q_Y(z, \xi) }{\log(1+|\xi|)} \ge  c_1   \lim_{u \to 0}\frac{ u^{-\alpha} }{\log(1+1/u)} =\infty.
\end{equation*}
Thus, by \cite[Theorem  1.2]{SW13} and the Fourier inversion theorem, $Y^S$ has a transition density function  $p_S(t,x,y)$ which is given by
\begin{align*}
	p_S(t,x,y) = (2\pi)^{-d} \int_{{\mathbb R}^d} e^{-i\la \xi, y-x\ra} \lambda_t(x,\xi) d \xi, \quad t>0,\; x,y \in {\mathbb R}^d.
\end{align*}
By \cite[Theorem 2.7]{SW13} and \eqref{e:H-W-condition}, we see that for all $y \in {\mathbb R}^d$,
\begin{align*}
	&	|p_S(t_1,x_0,x_0)-p_S(t_1,x_0,x_0+y)|\\
	&\le (2\pi)^{-d} \int_{{\mathbb R}^d} |1-e^{-i\la \xi, y\ra} | \lambda_{t_1}(x_0,\xi) d \xi \le (2\pi)^{-d} |y| \int_{{\mathbb R}^d} |\xi| \lambda_{t_1}(x_0,\xi) d \xi\\
	& \le (2\pi)^{-d} |y| \int_{{\mathbb R}^d} |\xi| \exp \bigg(-\frac{t_1}{8} \inf_{z \in {\mathbb R}^d} \text{Re} \, q_Y(z,\xi)\bigg) d \xi\\
	&\le (2\pi)^{-d} |y| \bigg( \int_{|\xi| \le \eta_0^{-1/\alpha}r_1^{-1}} |\xi| d \xi+ \int_{|\xi| >\eta_0^{-1/\alpha}r_1^{-1}} |\xi| \exp \bigg(-\frac{t_1}{8} \inf_{z \in {\mathbb R}^d} \text{Re} \, q_Y(z,\xi)\bigg) d \xi\bigg) \\
	& \le c_2 |y| \bigg(\eta_0^{-(d+1)/\alpha}r_1^{-(d+1)}+ \int_{\eta_0^{-1/\alpha}r_1^{-1}}^\infty s^d \exp \big( - c_3 \eta_0 r_1^\alpha s^\alpha \big) ds\bigg).
\end{align*}
Using the inequality  $e^{-s} \le r^r s^{-r}$ for all $s,r>0$, we obtain
\begin{align*}
	\int_{\eta_0^{-1/\alpha}r_1^{-1}}^\infty  s^d \exp \big( - c_3 \eta_0 r_1^\alpha s^\alpha \big) ds \le c_4\eta_0^{-(d+2)/\alpha}r_1^{-(d+2)}\int_{\eta_0^{-1/\alpha}r_1^{-1}}^\infty s^{-2} ds=c_5 \eta_0^{-(d+1)/\alpha}r_1^{-(d+1)}.
\end{align*}
Therefore, we deduce that
\begin{equation}\label{e:pY-gradient}
	|p_S(t_1,x_0,x_0)-p_S(t_1,x_0,x_0+y)| \le c_6 |y| \eta_0^{-(d+1)/\alpha}r_1^{-(d+1)} \quad \text{for all} \;\; y \in {\mathbb R}^d.
\end{equation}
On the other hand, similar to \eqref{e:H-W-condition}, using \eqref{a:as3}, \eqref{e:g-comp}, \eqref{a:as2}, \eqref{e:g-scaling} and Lemmas \ref{l:(S)} and \ref{l:Phi-2}, we get that  for all $u<2r_1$,
\begin{align*}
	\sup_{z \in {\mathbb R}^d} \sup_{|\xi|=1/u} 2 \text{Re} \, q_Y(z, \xi/2)& \le K_0^2 \sup_{z \in B(x_0,2r_1)} g(z, 1/u) \le   c_U K_0^2 \bigg( \frac{r_1}{u}\bigg)^\beta \sup_{z \in B(x_0,2r_1)} g(z, 1/r_1) \\
	& \le  c_U K_0^2 K_1 \bigg( \frac{r_1}{u}\bigg)^\beta \sup_{z \in B(x_0,2r_1)} \frac{1}{\Phi(z, r_1)} \\
	&\le \frac{c_U K_0^2 K_1}{C_{13}} \bigg( \frac{r_1}{u}\bigg)^\beta \frac{1}{\Phi(x_0, r_1)} =  \frac{c_U K_0^2 K_1}{C_{13}} \bigg( \frac{r_1}{u}\bigg)^\beta  \frac{\eta_0}{t_1}.
\end{align*}
Put $c_7:= c_U K_0^2 K_1/C_{13}>1$. By taking $\eta_0$ small enough, we may assume $4c_7\eta_0<1$. Then by the second display in \cite[p.3265]{SW13}, it holds that for all $2^{-1}r_1^{-1}<|\xi|<(4c_7\eta_0)^{-1/\beta}r_1^{-1}$,
\begin{equation*}
	\text{Re} \, \lambda_{t_1}(x_0, \xi) \ge 1- 2t_1 \sup_{z \in {\mathbb R}^d} 2\text{Re}\, q_Y(z, \xi/2) \ge 1-2c_7\eta_0r_1^\beta |\xi|^{\beta} \ge 2^{-1}.
\end{equation*}
Since $\lambda_{t_1}(x_0,\xi)\ge 0$ for every $\xi\in {\mathbb R}^d$, it follows that
\begin{align*}
	p_S(t_1, x_0,x_0) \ge (2\pi)^{-d} \int_{2^{-1}r_1^{-1}<|\xi|<(4c_7\eta_0)^{-1/\beta}r_1^{-1}} 2^{-1} d\xi \ge c_8 \eta_0^{-d/\beta}r_1^{-d}.
\end{align*}
Combining with \eqref{e:pY-gradient}, we obtain that for all $y \in {\mathbb R}^d$ with $ |y| \le 2^{-1}c_6^{-1}c_8 \eta_0^{ - d/\beta + (d+1)/\alpha} r_1$,
\begin{equation*}
	p_S(t_1,x_0, x_0+y) \ge c_8 \eta_0^{-d/\beta} r_1^{-d} -2^{-1}c_8 \eta_0^{-d/\beta} r_1^{-d} =2^{-1}c_8 \eta_0^{-d/\beta} r_1^{-d} .
\end{equation*}
Therefore, it holds that for every $z \in {\mathbb R}^d$ with $|z|=1$,
\begin{align*}
	&{\mathbb P}^{x_0}\big(2\la  Y^S_{t_1}- x_0, z \ra \le -| Y^S_{t_1} -x_0 |  \big)  \ge \int_{ 2 \la y, z\ra \le -|y|, \, |y| \le 2^{-1}c_6^{-1}c_8 \eta_0^{ - d/\beta + (d+1)/\alpha} r_1} p_S(t_1, x_0, x_0+y) dy\\
	& \ge  2^{-1}c_8 \eta_0^{-d/\beta} r_1^{-d}  \int_{ 2 \la y, z\ra \le -|y|, \, |y| \le 2^{-1}c_6^{-1}c_8 \eta_0^{ - d/\beta + (d+1)/\alpha} r_1} dy \\
	&=2^{-1}c_8 \eta_0^{-d/\beta} r_1^{-d}  \big(2^{-1}c_6^{-1}c_8 \eta_0^{ - d/\beta + (d+1)/\alpha} r_1 \big)^d \int_{ 2 \la y, z\ra \le -|y|, \, |y| \le 1} dy= c_9  \eta_0^{(d^2+d)(1/\alpha - 1/\beta)}
\end{align*}
and hence by \eqref{e:O4-main}, 
$$
	{\mathbb P}^{x_0}\big(2\la  X_{t_1}- x_0, z \ra \le -| X_{t_1} -x_0 |  \big) \ge  \eta_0 \big(  c_9\eta_0^{-1+(d^2+d)(1/\alpha - 1/\beta)} - C_{14} \big).
$$
Note that the above constants $c_9$ and $C_{14}$ are independent of $x_0$ and $t_1$. In view of \eqref{e:alphabeta}, we conclude \eqref{O4} by taking $\eta_0$  sufficiently small. \end{proof}

Below, we give two concrete examples. In the following  examples, we assume that   $U \subset {\mathbb R}^d$, $d \ge 1$ is an open set and $C_c^\infty({\mathbb R}^d)$ is an operator core for  the generator of the Feller process $X$.

\begin{example}\label{E:Feller}
	{\rm {\bf (Symbols of varying order)} Suppose that 	 there are  H\"older continuous functions $\alpha:U \to (0,2)$ and $\gamma:U \to (-1,1)$   such that  $\inf_{x \in U} \alpha(x)>0$, $ \alpha(x)/2 + \gamma(x) \in [0,1]$  for all $x \in U$, and that
		$$
		q(x, \xi)=|\xi|^{\alpha(x)}(\log (1+|\xi|))^{\gamma(x)} \quad \text{for all} \;\; x \in U, \; \xi \in {\mathbb R}^d.
		$$

		By H\"older continuities of $\alpha(x)$ and $\gamma(x)$, there exist constants  $c_1>0$ and $\theta \in (0,1]$ such that $|\alpha(x)-\alpha(y)|+|\gamma(x)-\gamma(y)|\le c_1|x-y|^\theta$ for all $x,y \in U$. Since $\lim_{r \to 0} r^{c_1 r^\theta}=\lim_{r \to 0} (\log(1+1/r))^{-c_1 r^\theta}=1$, we see that	for all $x \in U$ and $\xi\in {\mathbb R}^d$ with $r:=1/|\xi|<1 \wedge \updelta_U(x)$,  
		\begin{align*}
			\inf_{|x-y| \le r} \text{Re} \, q(y,\xi) &\ge  r^{-\alpha(x)}(\log (1+1/r))^{\gamma(x)}  	\inf_{|x-y| \le r} r^{\alpha(x)-\alpha(y)} (\log(1+1/r))^{\gamma(y)-\gamma(x)}  \\
			&\ge  r^{-\alpha(x)}(\log (1+1/r))^{\gamma(x)}    r^{ c_1 r^\theta} (\log(1+1/r))^{-c_1 r^\theta} \\[6pt]
			&\ge c_2r^{-\alpha(x)}(\log (1+1/r))^{\gamma(x)}   
		\end{align*}
		and 
		\begin{align*}
			\sup_{|x-y| \le r} \text{Re} \, q(y,\xi) 	&\le  r^{-\alpha(x)}(\log (1+1/r))^{\gamma(x)}    r^{ -c_1 r^\theta} (\log(1+1/r))^{c_1 r^\theta}
			\\
			& \le c_3r^{-\alpha(x)}(\log (1+1/r))^{\gamma(x)}.
		\end{align*}
		Hence, \eqref{O3} holds.  
		
		Now, we check that {\bf (S)} is fulfilled.  
		Define $g(x,r)=r^{\alpha(x)} (\log (1+r))^{\gamma(x)}$ for $x \in U$, $r>0$. Since for any $\eps>0$, there is a constant $c_4=c_4(\eps)>0$ such that
		\begin{equation*}	\frac{\log(1+r)}{\log(1+s)} \le c_4 \bigg(\frac{r}{s} \bigg)^{\eps} \quad \text{for all} \;\; r \ge s \ge 1,
		\end{equation*}	
	one can see that $g(x,\cdot)$ satisfies {\bf (S)}(i).
	
		 Next, fix  any $x_0 \in U$ and $0<2r<1 \wedge \updelta_U(x_0)$. Let $\wt \alpha:{\mathbb R}^d \to (0,2]$ and $\wt \gamma:{\mathbb R}^d \to (-1,1)$ be H\"older continuous functions  such that (i)  for every $x \in \overline{B(x_0,r)}$, $\wt\alpha(x)=\alpha(x)$ and $\wt \gamma(x)=\gamma(x)$ and (ii) for every $x \in {\mathbb R}^d \setminus \overline{B(x_0,r)}$,  $\wt \alpha(x)/2+\wt \gamma(x) \in [0,1]$ and  for all $u>16$, 
		\begin{equation*}
			\frac{1}{2}\inf_{|y-x_0| \le r} u^{\alpha(y)}(\log (1+u))^{\gamma(y)}	\le u^{\alpha(x)}(\log (1+u))^{\gamma(x)}\le 2\sup_{|y-x_0| \le r} u^{\alpha(y)}(\log (1+u))^{\gamma(y)}.
		\end{equation*} 
	According to \cite[Theorem 3.3 and Extension 3.13]{Ku}, there exists a Feller process $Y$ on ${\mathbb R}^d$  having the symbol $q_Y(x,\xi)=2^{\wt \alpha(x)-1} |\xi|^{\wt \alpha(x)} (\log (1+2|\xi|))^{\wt \gamma(x)}$. Hence {\bf (S)}(ii) holds. 
		
		Note that $\Phi(x,r)=r^{\alpha(x)} (\log ( 1+1/r))^{-\gamma(x)}$ for $x \in U$ and $r>0$ in this case.  Hence, 
		\begin{equation}\label{e:Phi-inverse}
			\lim_{t \to 0} \frac{\Phi^{-1}(x, t/\log|\log t|)}{t^{1/\alpha(x)} |\log t|^{\gamma(x)/\alpha(x)} (\log|\log t|)^{-1/\alpha(x)}}=\alpha(x)^{\gamma(x)} \quad \text{for all} \;\; x \in U.
		\end{equation}
		Finally, since $\inf_{y \in U} \alpha(y) \wedge 2^{-1} \le \alpha(x)^{\gamma(x)} \le 2$ for all $x \in U$,  using
 Proposition \ref{p:FellerO4}, Theorem \ref{t:Feller1} and \eqref{e:Phi-inverse}, we conclude that there are constants  
		$a_2 \ge a_1>0$ such that for all $x \in U$, there exists  a constant $a_x \in [a_{1}, a_{2}]$ such that
		\begin{equation}\label{E:Feller-result1}
			\liminf_{t\to0} \frac{ \sup_{0<s\le t} | X_s - x | }{t^{1/\alpha(x)} |\log t|^{\gamma(x)/\alpha(x)} (\log|\log t|)^{-1/\alpha(x)}}=a_x,~\quad
			{\mathbb P}^x\mbox{-a.s.}
		\end{equation} 
	}
\end{example}

Let $\mathbb S^{d-1}:=\{y \in {\mathbb R}^d: |y|=1\}$ and $\mathbf e_i=\mathbf e_i(d)$, $1\le i\le d$ denote the standard basis of  ${\mathbb R}^d$.

\begin{example}\label{E:singular}
	{\rm  {\bf (Cylindrical stable-like  processes)}  Suppose that  $d \ge 2$ and 	 there exists a  H\"older continuous function $\alpha:U \to (0,2)$  with $\inf_{x \in U} \alpha(x)>0$ such that 
		\begin{equation*}
			q(x, \xi)=\sum_{i=1}^d |\xi_i|^{\alpha(x)} \;\; \text{for all  $x \in U$ and $\xi=(\xi_1,...,\xi_d)\in {\mathbb R}^d$}.
		\end{equation*}
		Note that for every $x \in U$, the L\'evy measure $\nu(x,dz)$ is a stable  kernel of the form
		\begin{equation}
			\nu(x,A) = \frac{\alpha(x) 2^{\alpha(x)-1} \Gamma((1+\alpha(x))/2)}{\pi^{1/2}\Gamma(1-\alpha(x)/2)} \int_0^\infty \int_{\mathbb S^{d-1}} \1_A(r \theta) r^{-1-\alpha(x)} \sum_{i=1}^d \delta_{\{\mathbf e_i\}}(\theta) dr,
		\end{equation}
		where $\Gamma(z):=\int_0^\infty u^{z-1}e^{-u}du$ is the gamma function and $\delta_{\{\mathbf e_i\}}$ is a Dirac measure on $\{\mathbf e_i\}$. Since  $|\xi|^{\alpha(x)} \le q(x,\xi) \le d|\xi|^{\alpha(x)}$ for all $x \in U$ and $\xi\in {\mathbb R}^d$, using the H\"older continuity of $\alpha$, one can see that \eqref{O3} holds  as in Example \ref{E:Feller}. Clearly,  {\bf (S)}(i) holds with  $g(x,r)=r^{\alpha(x)}$. 	 Choose  any $x_0 \in U$ and $0<2r<1 \wedge \updelta_U(x_0)$, and let $\wt \alpha:{\mathbb R}^d \to (0,2)$ be a H\"older continuous function  such that  for every $x \in \overline{B(x_0,r)}$, $\wt\alpha(x)=\alpha(x)$ and  for every $x \in {\mathbb R}^d \setminus \overline{B(x_0,r)}$, 
		\begin{equation*}
			\frac{1}{2}\inf_{|y-x_0| \le r} u^{\alpha(y)} 	\le u^{\alpha(x)} \le 2\sup_{|y-x_0| \le r} u^{\alpha(y)} \;\; \text{for all} \; u>16.
		\end{equation*} 
		According to \cite[Theorem 3.1]{KKS}, since the measure  $\sum_{i=1}^d \delta_{\{\mathbf e_i\}}$  on $\mathbb S^{d-1}$ is nondegenerate in the sense of \cite[{\bf (M1)}]{KKS}, there exists a Feller process $Y$ on ${\mathbb R}^d$ having the symbol $q_Y(x,\xi)=2^{\wt \alpha(x)-1} \sum_{i=1}^d |\xi_i|^{\wt \alpha(x)}$. Thus,  {\bf (S)}(ii) is satisfied. 
		
		In the end, using 
		Proposition \ref{p:FellerO4} and  Theorem \ref{t:Feller1} again, 
		we get a similar equation to \eqref{e:Phi-inverse} and  we can deduce that for all $x \in U$, the  LIL \eqref{E:Feller-result1} holds  with $\gamma=0$. 
	}
\end{example}

\section{Liminf LILs at infinity for random conductance model with long range jumps}\label{s:RCM}

In  \cite[Section 3]{CKL}, we have obtained limsup LILs at infinity for random conductance model with long range jumps using results in \cite{CKW18, CKW20-1}. In this section, we give liminf LILs  at infinity for such models.  We repeat the setting of the  random conductance models in  \cite[Section 3]{CKL} here 
for the readers' convenience. 

	Let $G= (\mathbb{L}, E_\mathbb{L})$ be a locally finite connected infinite undirected graph, where $\mathbb{L}$ is the set of vertices, and $E_\mathbb{L}$ the set of edges. For $x,y \in \mathbb{L}$, we denote $d(x,y)$ for the graph distance, namely, the length of the shortest path joining $x$ and $y$. Let $\mu_c$ be the counting measure on $\mathbb{L}$.   We  assume that for some constant $d>0$,
	\begin{equation}\label{(3.1)}
		\mu_c(B(x,r)) \asymp r^d \quad \mbox{for} \,\, x \in \mathbb{L}, \, r>10
	\end{equation}
	
	A \textit{random conductance} $\boldsymbol{\eta}=(\eta_{xy} : x,y \in \mathbb{L})$ on $\mathbb{L}$ is a family of nonnegative random variables defined on some probability space $(\Omega,{\bf F}, {\bf P})$ such that  $\eta_{xx} = 0$ and $\eta_{xy} = \eta_{yx}$ for all $x,y \in \mathbb{L}$. We set $\nu_x:=\sum_{y \in \mathbb{L}} \eta_{xy}$ for $x \in \mathbb{L}$ and  denote $\mathbf E$ for the expectation with respect to $\mathbf{P}$.  
	 For each $\omega \in \Omega$, the \textit{variable speed random walk} (VSRW) $X^{\omega} = (X^\omega_t, t \ge 0; {\mathbb P}^x_\omega, x \in \mathbb{L})$ (associated with  $\boldsymbol{\eta}$) is defined by the symmetric Markov process on $\mathbb{L}$ with $L^2(\mathbb{L}, \mu_c)$-generator
		$$ \sL_V^\omega f(x) = \sum_{y \in \mathbb{L}} \eta_{xy}(\omega)(f(y) - f(x)), \quad x \in \mathbb{L}, $$
		and the \textit{constant speed random walk} (CSRW) $Y^\omega = (Y^\omega_t, t \ge 0; {\mathbb P}^x_\omega, x \in \mathbb{L})$ (associated with $\boldsymbol{\eta}$)  is the symmetric Markov process on $\mathbb{L}$ with $L^2(\mathbb{L}, \nu)$-generator
		$$ \sL_C^\omega f(x) = \nu_x^{-1}(\omega) \sum_{y \in \mathbb{L}} \eta_{xy}(\omega)(f(y)-f(x)), \quad x \in \mathbb{L}. $$
		
		Let $\alpha \in (0,2)$ and  $\boldsymbol{\eta}$ be a random conductance on $\mathbb{L}$. With the constant $d>0$ in \eqref{(3.1)}, we write $w_{xy}:=\eta_{xy} |x-y|^{d+\alpha}$ for $x,y \in \mathbb{L}$ so that 
		\begin{equation*}
		\eta_{xx}=w_{xx} = 0 \quad \mbox{and} \quad \eta_{xy}(\omega) = \frac{w_{xy}(\omega)}{|x-y|^{d+\alpha}}, \quad x \neq y, \quad x,y \in \mathbb{L},
	\end{equation*}
		Suppose that $d > 4-2\alpha$, 
$$
			\sup_{x,y \in \mathbb{L}, x\neq y} {\bf P}(w_{xy} = 0)=\sup_{x,y \in \mathbb{L}, x\neq y} {\bf P}(\eta_{xy} = 0)< 1/2 ,
			$$ 
	and 
	$$\sup_{x,y \in \mathbb{L}, x\neq y} \big( {\bf E}[w_{xy}^p] + {\bf E}[w_{xy}^{-q} \1_{\{ w_{xy} > 0  \}}] \big) < \infty$$
			with some constants
		\begin{equation*}
			p> \frac{d+2}{d} \lor \frac{d+1}{4-2\alpha} \quad  \mbox{and} \quad q > \frac{d+2}{d}.
		\end{equation*}
	When we consider the CSRW $Y^\omega$, we also assume that there exist constants $m_2 \ge m_1 >0$ such that for ${\bf P}$-a.s. $\omega$,
				\begin{equation*}
			\eta_{xy}(\omega) > 0 \;\; \mbox{for all} \; x,y \in \mathbb{L}, \, x \neq y \quad \mbox{and} \quad m_1 \le \sum_{y \in \mathbb{L}, y \neq x} \eta_{xy}(\omega)\le m_2 \;\; \mbox{for all} \; x \in \mathbb{L}.
		\end{equation*}
	
		According to the proof of  \cite[Theorem 3.1]{CKL}, for ${\bf P}$-a.s. $\omega$, there are constants $\ups \in (0,1)$ independent of $\omega$ and  $r_1(\omega), r_2(\omega) \ge 1$    such that    conditions 	${\rm Tail}^{r_1(\omega)}(r^\alpha,\ups)$ and ${\rm NDL}^{r_2(\omega)}(r^\alpha,\ups)$ (in Definitions \ref{d:inf})
		  hold for both $X^\omega$ and $Y^\omega$. Then, since ${\rm VRD}^{10}(\ups)$ holds by \eqref{(3.1)},  using Lemma \ref{l:NDL}(ii), we conclude from  Corollary \ref{c:infinf}   that 
		there exist constants $0<a_1 \le a_2 <\infty$ such that for  $\mathbf P$-a.s. $\omega$, there exist $a_3(\omega),a_4(\omega) \in [a_1,a_2]$ so that for all $x,y \in \bL$,
		\begin{align*}
			\liminf_{t \to \infty} \frac{ \sup_{0<s \le t} d(x,X^\omega_s)}{t^{1/\alpha} (\log\log t)^{-1/\alpha}} = a_3(\omega), \quad  \liminf_{t \to \infty} \frac{ \sup_{0<s \le t} d(x,Y^\omega_s)}{t^{1/\alpha} (\log\log t)^{-1/\alpha}} = a_4(\omega), \quad\;\, {\mathbb P}^y_\omega\text{-a.s.}
		\end{align*}
	Moreover, when $\alpha \in (0,1)$, the above LILs still hold true for $d > 2-2\alpha$, if $p>\max\{ (d+2)/d,  (d+1)/(2-2\alpha) \}$ and $q>(d+2)/d$, by the proof for the latter statement of \cite[Theorem 3.1]{CKL}.
		
		\medskip

	\section{Liminf LILs for subordinate processes and symmetric Hunt processes}\label{s:hunt}
	In this section, we give  liminf LILs for subordinate processes and symmetric Hunt processes. See \cite[Section 2]{CKL} for detailed descriptions and limsup LILs for such processes.
	
	\medskip
	
		Recall that $(M,d,\mu)$ is a locally compact separable metric space with a positive Radon measure $\mu$ on $M$ with full support. Let $\bar R := \sup_{y,z \in M} d(y,z)$ and $F$ be an increasing and continuous function on $(0,\infty)$ such that for some constants $\gamma_2 \ge \gamma_1 > 1$ and $c_L,c_U>0$,
		\begin{equation}\label{e:scaling-F} c_L \bigg( \frac{R}{r} \bigg)^{\gamma_1} \le \frac{F(R)}{F(r)} \le c_U \bigg( \frac{R}{r} \bigg)^{\gamma_2} \quad \mbox{for all} \,\, 0<r \le R < \bar R. 
			\end{equation}
		We assume that  ${\rm VRD}_{\bar R}(M)$ and a chain condition ${\rm Ch}_{\bar R}(M)$ (see \cite[Definition 1.2]{CKL}) hold. 		 We also assume that there exists a conservative Hunt process $Z = (Z_t, t \ge 0; {\mathbb P}^x, x \in M)$ on $M$ whose heat kernel $q(t,x,y)$ (with respect to $\mu$) exists and satisfies the following estimates: There are constants $R_1 \le \bar R$ and $c_1,c_2,c_3>0$ such that for all $t \in (0,F(R_1))$ and $x,y \in M$,
			\begin{equation}\label{e:heatkernel-F}
				\frac{c_1}{V(x,F^{-1}(t))} \1_{\{F(d(x,y)) \le t\}} \le q(t,x,y) \le \frac{c_2}{V(x,F^{-1}(t))} \exp \big( - c_3 F_1(d(x,y),t) \big),
			\end{equation}
			where the function $F_1$ is defined by		 $\displaystyle  F_1(r,t) := \sup_{s >0} \Big( \frac{r}{s} - \frac{t}{F(s)} \Big)$.
		   
		   \smallskip

		   Let $S = (S_t)_{t \ge 0}$ be a subordinator independent of $Z$. We denote by $\phi_1$ the Laplace exponent of $S$. Then it is well known that there exist a constant $b \ge 0$ and a Borel measure $\nu$ on $(0,\infty)$ satisfying $\int_0^\infty (1 \wedge u) \nu(du)<\infty$ such that 
		   $$ \phi_1(\lambda):=-\log \E[e^{-\lambda S_1}] = b \lambda + \int_{(0,\infty)}(1- e^{-\lambda u})\nu(du), \quad\;\, \lambda>0. $$
		  We assume either $b \neq 0$ or $\nu((0,\infty)) = \infty$. 
		   
		   Let $X=(X_t)_{t\ge 0}$ be the subordinate process defined by	   $X_t := Z_{S_t}$.  Define
		   $$ \Phi(r)= \frac{1}{\phi_1(1/F(r))} \quad \mbox{ and } \quad \Pi(r)= \frac{2e}{\nu((F(r),\infty))} \quad\;\; \mbox{for} \;\, r>0.$$
	Then $\Phi$ and $\Pi$ are nondecreasing and $ \Phi(r)  \le \Pi(r)$ for all $r>0$. Moreover, since we have  assumed that either $b\neq 0$ or $\nu((0,\infty))=\infty$,  by  \eqref{e:scaling-F}, we have that  $\lim_{r \to 0} \Phi(r)=0$ and that
	\begin{equation}\label{e:scaling-Phi}
		\frac{\Phi(r)}{\Phi(s)} \le c_U \bigg( \frac{r}{s}\bigg)^{\gamma_2} \quad \mbox{for all} \,\, 0<r \le R < \bar R. 
	\end{equation}
See \cite[Subsection 2.1]{CKL}. 

By \eqref{e:scaling-Phi} and \cite[Lemmas A.2 and A.3(i)]{CKL},  using Lemma \ref{l:NDL}(i), we see that conditions \eqref{A1}-\eqref{A4} hold  for $U = M$ and that  the function $\phi(x,r):=\Phi(r)$ satisfies  \eqref{e:phi} for $0<r\le 1$. Therefore, we get from  Theorem \ref{inf} and Corollary \ref{c:inf} that
\begin{thm}
There are constants $a_2 \ge a_1>0$ such that for all $x \in M$, there exists a constant $a_x \in [a_1,a_2]$ satisfying
	\begin{equation}\label{e:Hunt-zero-pre}
		\liminf_{t\to0} \frac{\Phi\big(\sup_{0<s\le t} d(x, X_s)\big)}{t/\log|\log t|}=a_x,\qquad
		{\mathbb P}^x\mbox{-a.s.}
	\end{equation}
Moreover, if $\phi_1$ satisfies lower scaling property ${\mathrm{L}}^1(\phi_1, \beta_1, c_1)$ (see Definition \ref{d:ws} in Appendix) for some $\beta_1,c_1>0$, then there are constants $\wt a_2 \ge \wt a_1>0$  such that for all $x \in M$, there exists a constant $\wt a_x \in [\wt a_1,\wt a_2]$ satisfying
\begin{equation}\label{e:Hunt-zero}
	\liminf_{t\to0} \frac{\sup_{0<s\le t} d(x, X_s)}{\Phi^{-1}(t/\log|\log t|)}=\wt a_x,\qquad
	{\mathbb P}^x\mbox{-a.s.}
\end{equation}
\end{thm}

\medskip

Here, we point out that our liminf LIL \eqref{e:Hunt-zero-pre}  covers the cases when $\phi_1$ is slowly varying at infinity. Therefore, general liminf LIL \eqref{e:Hunt-zero-pre}   can be applicable to some jump processes with low intensity of small jumps such as geometric $2\alpha$-stable processes on ${\mathbb R}^d$ ($0<\alpha\le 1$), namely, a L\'evy process on ${\mathbb R}^d$ with the characteristic exponent $\log(1+|\xi|^{2\alpha})$.

\medskip

To get liminf LILs at infinity, we also assume that constants $\bar R = R_1=\infty$ in \eqref{e:scaling-F} and \eqref{e:heatkernel-F}, and ${\mathrm{L}}_1(\phi_1, \beta_1, c_1)$ (see Definition \ref{d:ws} in Appendix \ref{s:A}) hold for some $\beta_1, c_1>0$.
Then by  \cite[Lemma A.4(ii)]{CKL} and Lemma \ref{l:NDL}(ii), the function $\phi(x,r):=\Phi(r)$ satisfies  \eqref{e:phi} for $r\ge 1$ and \eqref{B4+} holds true. Since conditions \eqref{B1}-\eqref{B3} holds  by \cite[Lemmas A.2 and A.3(i)]{CKL}, we deduce from  Corollary \ref{c:infinf} that 
\begin{thm}
Suppose that \eqref{e:scaling-F} and \eqref{e:heatkernel-F} hold true with $\bar R = R_1=\infty$, and $\phi_1$ satisfies the lower scaling property ${\mathrm{L}}_1(\phi_1, \beta_1, c_1)$  for some $\beta_1, c_1>0$, i.e., 
$$ \frac{\phi_1(r)}{\phi_1(s)} \geq c_1 \Big(\frac{r}{s}\Big)^{\beta_1} \quad \text{for all} \quad s\leq r< 1.$$
Then,	there  exists a constant $b_\infty$ such that for all $x,y \in M$, 
	\begin{equation}\label{e:Hunt-infty}
		\liminf_{t\to\infty} \frac{\sup_{0<s\le t} d(x, X_s)}{\Phi^{-1}(t/\log\log t)}=b_\infty,\qquad
		{\mathbb P}^y\mbox{-a.s.}
	\end{equation}
\end{thm}

\bigskip

Similar results hold for symmetric Hunt processes considered in \cite[Subsection 2.2]{CKL}. Precisely, let $X$ be a Hunt process on $M$ associated with a regular Dirichlet form $(\sE^X, \sF^X)$ of the form \cite[(2.21)]{CKL} satisfying \cite[Assumption L]{CKL}. With the function $\Phi_1$ defined in \cite[(2.22)]{CKL} and open subset $\sU$ of $M$ in \cite[Assumption L]{CKL}, using \cite[(B.7), Propositions B.1 and B.13 and Lemma B.8]{CKL} and our Lemma \ref{l:NDL}, we can deduce that  the function $\phi(x,r):=\Phi_1(r)$ satisfies  \eqref{e:phi} for $0<r\le 1$, and conditions \eqref{A1}-\eqref{A3} and \eqref{A4+} hold for $U=\sU$.  Moreover, when constants $\bar R = R_1=\infty$ in \eqref{e:scaling-F} and \eqref{e:heatkernel-F}, the function $\phi(x,r):=\Phi_1(r)$ satisfies  \eqref{e:phi} for $r\ge 1$, and conditions \eqref{B1}-\eqref{B3} and \eqref{B4+} hold. Therefore, we conclude from Corollaries \ref{c:inf} and \ref{c:infinf} that the liminf LIL at zero \eqref{e:Hunt-zero} holds for $x\in \sU$ with the function $\Phi_1$ instead of $\Phi$, and if we also assume $\bar R = R_1=\infty$ in \eqref{e:scaling-F} and \eqref{e:heatkernel-F}, then the liminf LIL at infinity \eqref{e:Hunt-infty} holds with the function $\Phi_1$ instead of $\Phi$.

\section{Proof of Main theorems}\label{s:proof}

Recall that we always assume that $\phi(x,r)$ (and $\phi(r)$) satisfies \eqref{e:phi}.

Proposition \ref{p:EP}(ii) below  follows from  \cite[Proposition 4.9(ii) and Corollary 4.10]{CKL}. 
Moreover, when $r\mapsto \phi(x,r)$ is comparable with a strictly increasing continuous function on $(0,\infty)$ independent of $x \in U$, the inequality \eqref{e:EP+} of Proposition \ref{p:EP}(i) is obtained in 
\cite[Proposition 4.9(i)]{CKL}
 with $\theta=1$. But since we allow $\phi(x,r)$ to depend on the space variable $x$ here, we need some significant modifications in the proof for the next proposition.

\begin{prop}\label{p:EP} 
	\noindent (i) Suppose that \eqref{A1}, \eqref{A2} and \eqref{A3} hold. Then there exist constants $\theta \in (0,1]$ and $c>0$ such that  for all $x \in U$, $0<r<3^{-1}( R_0 \wedge (C_0\updelta_U(x)))$ and $t>0$,
	\begin{equation}\label{e:EP+}
		{\mathbb P}^x(\tau_{B(x,r)} \le t) \le c \left(\frac{t}{\phi(x,r)}\right)^{\theta}.
	\end{equation}
	
	\noindent (ii) Suppose that \eqref{B1}, \eqref{B2} and \eqref{B3} hold. Let $\ups_1 \in (\ups, 1)$. Then there exist constants $c>0$ and  $R_1 \ge R_\infty$  such that \eqref{e:EP+} holds with $\theta=1$ and $\phi(r)$ instead of $\phi(x,r)$ for all $x \in M$, $r>R_1 {\mathsf{d}}(x)^{\ups_1}$ and $t \ge \phi(2r^{\ups/\ups_1})$.  	Moreover, $X$ is conservative, that is, ${\mathbb P}^x(\zeta=\infty)=1$  for all $x \in M$.
\end{prop}

 Before giving the proof of Proposition \ref{p:EP}, we present some lemmas which will be used in the proof of Proposition \ref{p:EP}(i).

For $\rho>0$, let $X^{(\rho)}$ be a Borel standard Markov process on $M$
obtained from $X$ by suppressing all jumps with jump size bigger than $\rho$
so that the L\'evy measure $J^{(\rho)}(x,dy)$ of $X^{(\rho)}$ is $J^{(\rho)}(x,A)= J(x,A \cap B(x,\rho))$ for every measurable set $A \subset M$. Then the original process $X$ can be constructed from $X^{(\rho)}$ by the Meyer's construction. See \cite{Me75}  and  \cite[Section 3]{BGK09} for details.

Denote $\tau^{(\rho)}_D:=\inf\{t>0:X_t^{(\rho)} \in M_\partial \setminus D\}$ for the first exit time of $X^{(\rho)}$ from $D$.
We first generalize \cite[Lemma 4.7(i)]{CKL}.

\begin{lem}\label{l:3.8} 
Suppose that \eqref{A1}, \eqref{A2} and \eqref{A3} hold. Then, there exist constants $\delta \in (0,1)$ and $K_1>0$ such that for all $x \in U$ and $0<\rho<3^{-1}\big( R_0 \wedge (C_0 \updelta_U(x))\big)$,
	\begin{equation*}
		\E^x \bigg[ \exp \Big( - \frac{K_1 }{\E^x[\tau_{B(x,r)}]}  \tau^{(\rho)}_{B(x,\rho)} \Big)\bigg] \le 1-\delta. 
	\end{equation*}
\end{lem}
\begin{proof}  Let $x \in U$ and $0<\rho<3^{-1}\big( R_0 \wedge (C_0 \updelta_U(x))\big)$, and denote $\psi(x,r)= \E^x[\tau_{B(x,r)}]$. We follow the proof of \cite[Lemma 4.7(i)]{CKL}. By \eqref{A1} and \eqref{A2}, we have
$$
\sup_{z\in B(x,\rho)} \E^z\tau_{ B(x,\rho)} \le \sup_{z\in  B(x,\rho)} \E^z\tau_{B(z,2\rho)} \le c_1\sup_{z\in  B(x,\rho)} \E^z\tau_{B(z,\rho)}  \le c_2
\psi(x,\rho).
$$
Thus, by the same argument as that of \cite[(4.17)]{CKL}, there exist constants $c_3,c_4>0$ such that
\begin{equation}\label{e:l.3.8-1}
	{\mathbb P}^x(\tau_{ B(x,\rho)}>t) \ge c_3-\frac{c_4t}{
\psi(x,\rho)}, \quad t>0.
\end{equation} 
Moreover, by following the proof for  \cite[(4.18)]{CKL}, using \eqref{A1}, one can deduce that
\begin{equation}\label{e:l.3.8-2}
	\Big| {\mathbb P}^x(\tau_{ B(x,\rho)}>t) - {\mathbb P}^x(\tau_{ B(x,\rho)}^{(\rho)}>t) \Big| \le \frac{c_5t}{
\psi(x,\rho)}, \quad t>0.
\end{equation}

Let  $\delta=c_3/3$, $K_1=(c_4+c_5)\delta^{-1}\log(\delta^{-1})$ and $t_\rho=\delta 
\psi(x,\rho)/(c_4+c_5)$. Then by  \eqref{e:l.3.8-1} and \eqref{e:l.3.8-2},
\begin{align*}
	{\mathbb P}^x(\tau^{(\rho)}_{ B(x,\rho)}\le t_\rho)&=1 - {\mathbb P}^x( \tau_{ B(x,\rho)}> t_\rho) + {\mathbb P}^x( \tau_{ B(x,\rho)} > t_\rho)- {\mathbb P}^x( \tau_{ B(x,\rho)}^{(\rho)} > t_\rho) \\
	& \le 1-3\delta + \frac{(c_4+c_5)t_\rho}{
\psi(x,\rho)} = 1-2\delta.
\end{align*}
Hence, by the choice of $K_1$, we get that
\begin{align*}
	&\E^x \bigg[ \exp \Big( - \frac{K_1 }{
\psi(x,\rho)}  \tau^{(\rho)}_{B(x,\rho)} \Big)\bigg] \le \E^x \bigg[ \exp \Big( - \frac{K_1 }{
\psi(x,\rho)}  \tau^{(\rho)}_{B(x,\rho)} \Big) :  \tau_{ B(x,\rho)}^{(\rho)} \le t_\rho \bigg] \\
&\quad  +  \exp \Big( - \frac{K_1t_\rho }{
\psi(x,\rho)}   \Big)  
	\le {\mathbb P}^x( \tau_{ B(x,\rho)}^{(\rho)} \le t_\rho)  +  \exp \Big( - \frac{\delta K_1}{c_4+c_5}  \Big)\le 1 - 2\delta + \delta = 1-\delta.
\end{align*}
\end{proof}

Unlike \cite[Lemma 4.8(i)]{CKL}, we only get some polynomial bounds in the next lemma. But it is enough to prove Proposition \ref{p:EP} below.

\begin{lem}\label{l:3.9}
Suppose that \eqref{A1}, \eqref{A2} and \eqref{A3} hold.  Then, there exist constants $a_1,\theta_1>0$ such that for all $x \in U$ and $0<\rho \le r < 3^{-1}\big( R_0 \wedge (C_0 \updelta_U(x))\big)$,
	\begin{equation}\label{e:l.3.9}
		\E^x \bigg[ \exp\Big(-\frac{C_1K_1}{\E^x[\tau_{B(x,\rho)}]}\tau_{B(x,r)}^{(\rho)}  \Big)\bigg] \le a_1 \bigg( \frac{\rho}{r}\bigg)^{\theta_1},
	\end{equation}
where $C_1>0$ and $K_1>0$ are constants in \eqref{A1} and Lemma \ref{l:3.8} respectively.
\end{lem}
\begin{proof} 
By taking $a_1$ larger than $6^{\theta_1}$ in \eqref{e:l.3.9}, we may  assume that $6\rho \le r$ without loss of generality. 
Fix $x \in U$ and $0<6 \rho \le r < 3^{-1}\big( R_0 \wedge (C_0 \updelta_U(x))\big)$, and let  $
\psi(x,s)=\E^x[\tau_{B(x,s)}]$ for $s>0$. 
 Let $\lambda=C_1K_1/
\psi(x,\rho)$, $\tau_0=\tau^{(\rho)}_{B(x,r)}$, $u(z)=\E^z[e^{-\lambda \tau_0}]$ and
\begin{equation*}
	B_k=B(x, (2^k-1)\rho) \quad  \text{and} \quad b_k= \sup_{y \in B_k} u(y), \quad k \ge 1.
\end{equation*}
Fix any $\delta' \in (0,\delta)$ where $\delta \in(0,1)$ is the constant in Lemma \ref{l:3.8}. For each $k \ge 1$, let $z_k \in B_k$ be a point such that $u(z_k) \ge (1-\delta') b_k$ and $\tau_k:=\tau^{(\rho)}_{B(z_k,(2^k-1)\rho)}$. Since jump sizes of  $X^{(\rho)}$ are at most $\rho$, it holds that either $X^{(\rho)}_{\tau_k} \in B(z_k, 2^k\rho) \subset B_{k+1}$ or $X_{\tau_k}^{(\rho)}=\partial$. Therefore by the strong Markov property, we have that for all $k \ge 1$,
\begin{align}\label{e:l.3.9-1}
	(1-\delta')b_k &\le u(z_k) = \E^{z_k}[e^{-\lambda \tau_0} \, ;\, \tau_0<\zeta]= \E^{z_k} [ e^{-\lambda \tau_k} e^{-\lambda(\tau_0-\tau_k)} \, ; \, \tau_k \le \tau_0<\zeta] \nn\\
	&= \E^{z_k} [ e^{-\lambda \tau_k}\, \E^{X^{(\rho)}_{\tau_k}} [e^{-\lambda \tau_0}] \, ; \, \tau_k \le \tau_0<\zeta] \le b_{k+1} \E^{z_k}[e^{-\lambda \tau_k}].
\end{align}

Let $n_0 \in \N$ be such that $(2^{n_0}-1)\rho \le r/3 < (2^{n_0+1}-1)\rho$. Using the monotonicity of $
s \mapsto \psi(x,s)$,  \eqref{A1} and Lemma \ref{l:3.8}, since $z_k \in B_k$, we get that for $k \le n_0$,
\begin{align*}
	\E^{z_k}[e^{-\lambda \tau_k}]& \le  \E^{z_k} \bigg[ \exp \Big( - \frac{C_1K_1 }{
\psi(x, (2^k-1)\rho)}  \tau^{(\rho)}_{B(z_k,(2^k-1)\rho)} \Big)\bigg]  \\
	&\le  \E^{z_k} \bigg[ \exp \Big( - \frac{K_1}{
\psi(z_k, (2^k-1)\rho)}  \tau^{(\rho)}_{B(z_k,(2^k-1)\rho)} \Big)\bigg] \le 1-\delta.
\end{align*}
Combining  with \eqref{e:l.3.9-1}, we conclude that 
\begin{align*}
	u(x) &\le b_1 \le \frac{1-\delta}{1-\delta'} b_2 \le... \le \left(\frac{1-\delta}{1-\delta'} \right)^{n_0} b_{n_0+1} \le \left(\frac{1-\delta}{1-\delta'} \right)^{n_0}  \le \frac{1-\delta'}{1-\delta} \left(\frac{3\rho}{r}\right)^{\log\frac{1-\delta'}{1-\delta}/\log 2}.
\end{align*}
\end{proof}

\noindent \textbf{Proof of Proposition \ref{p:EP}.} 
(i) It suffices to prove for the case when  $\phi(x,r)= \E^x[\tau_{B(x,r)}]$ in view of \eqref{e:phi}.  We follow the proof of \cite[Proposition 4.9(i)]{CKL}, but  with nontrivial modifications.

 Choose any  $x \in U$, $0< r < 3^{-1}\big( R_0 \wedge (C_0 \updelta_U(x))\big)$ and $t>0$. Let  $\beta_2$ and $C_U$ be the constants from \eqref{e:scaling-zero}. If $t\ge C_U^{-1} 4^{-2\beta_2}\phi(x,r)$, then by taking $c$ larger than $C_U4^{2\beta_2}$, \eqref{e:EP+}  holds true.  Thus, we assume that $t<C_U^{-1} 4^{-2\beta_2}\phi(x,r)$.

  Set $\rho:=r (C_Ut/\phi(x,r))^{1/(2\beta_2)}$. Then  $\rho \in [\phi^{-1}(x,t), r/4)$. Indeed, 
 since we have assumed $t<C_U^{-1} 4^{-2\beta_2}\phi(x,r)$, by \eqref{e:scaling-zero}, it holds that
  \begin{equation*}
  r/4 > \rho \ge r \left(\phi^{-1}(x,t)/r\right)^{1/2} = r^{1/2} \phi^{-1}(x,t)^{1/2} \ge \phi^{-1}(x,t).
  \end{equation*} 
Using \eqref{e:scaling-zero} and \eqref{A1}, we see that for every $z \in B(x,2r)$, 
\begin{equation}\label{e:z-rho}
\frac{1}{\phi(z,\rho)} \le \frac{C_U2^{\beta_2}}{\phi(z,2r)}\left( \frac{r}{\rho}\right)^{\beta_2} \le \frac{C_1C_U2^{\beta_2}}{\phi(x,2r)}\left( \frac{r}{\rho}\right)^{\beta_2}\le \frac{C_1C_U2^{\beta_2}}{\phi(x,r)}\left( \frac{r}{\rho}\right)^{\beta_2}.
\end{equation}
Define $J_1(x,dy)=J(x,dy) \1_{\{\rho \le d(x,y) <r/4\}}$ and $J_2(x,dy)=J(x,dy) \1_{\{d(x,y) \ge r/4\}}$. Then we get from 
 \eqref{A3} and \eqref{e:z-rho} that
\begin{equation}\label{e:p.3.5-1}
	\sup_{z \in B(x,r)} J_1(z, M_\partial) \le \sup_{z \in B(x,r)} \frac{C_3}{\phi(z,\rho)}  \le  \frac{c_1}{\phi(x,r)}  \left( \frac{r}{\rho}\right)^{\beta_2}.
\end{equation}
We also get from \eqref{A3}, \eqref{e:scaling-zero} and \eqref{A1} that
\begin{equation}\label{e:p.3.5-2}
	\sup_{z \in B(x,r)} J_2(z, M_\partial) \le \sup_{z \in B(x,r)} \frac{C_3}{\phi(z,r/4)} \le  \sup_{z \in B(x,r)} \frac{C_3C_U4^{\beta_2}}{\phi(z,r)}  \le  \frac{c_2}{\phi(x,r)}.  
\end{equation}
As in \cite{CKL}, we let $Y^1:=X^{(\rho)}$, $Y^2$ be a Markov process obtained from $Y^1$ by attaching jumps coming from $J^1(x,dy)$, and $Y^3$ be a Markov process obtained from $Y^2$ by attaching jumps coming from $J^2(x,dy)$. For $n \ge 1$, denote by $T^1_n$ and $T^2_n$ the time at which $n$-th extra jump attached to $Y^1$ and $Y^2$, respectively. Let $\wt \tau_{B(x,r)}:= \inf\{u>0:Y^3_u \in M_\partial \setminus B(x,r)\}$. By the Meyer's construction, the law of $(Y^3_s: s<\tau_B)$ is the same as that of $(X_s: s<\tau_B)$. Therefore, it holds that
\begin{align}\label{e:4.30}
	&{\mathbb P}^x(\tau_{B(x,r)} \le t) = {\mathbb P}^x(\wt \tau_{B(x,r)} \le t)\nn\\
	&=  {\mathbb P}^x(T^1_{2} \le \wt \tau_{B(x,r)} \le t, \,\wt \tau_{B(x,r)} < T^2_{1}  ) + {\mathbb P}^x(T^2_{1} \le \wt \tau_{B(x,r)} \le t) \nn\\
	&\qquad+  {\mathbb P}^x(\wt \tau_{B(x,r)} \le t, \, \wt \tau_{B(x,r)} <T^1_{2} \wedge T^2_{1} )=:I_1+I_2+I_3.
\end{align} 

Let $Z_1,Z_2$ and $Z_3$ be i.i.d. exponential random variables with rate parameter $1$. From the Meyer's construction,  using \eqref{e:p.3.5-1} and  \eqref{e:p.3.5-2}, respectively, we get that
$$
I_1 \le {\mathbb P}\big( \frac{c_1t}{\phi(x,r)} \left( \frac{r}{\rho} \right)^{\beta_2} \ge Z_1 + Z_2 \big)\le \frac{c_1^2\,t^2}{\phi(x,r)^2}  \left( \frac{r}{\rho}\right)^{2\beta_2}
$$
and
$$
I_2 \le {\mathbb P}\big( \frac{c_2t}{\phi(x,r)} \ge Z_3\big) \le \frac{c_2t}{\phi(x,r)}.
$$
 On the event $\{\wt\tau_{B(x,r)} \le t, \, \wt\tau_{B(x,r)} <  T^1_{2} \wedge T^2_{1} \}$,  using the triangle inequality, we see that
 \begin{align}\label{e:I3-jump}
 	r \le d(x, Y^3_{\wt \tau_{B(x,r)}})  &\le d(x, Y^3_{ \wt \tau_{B(x,r)}\land T^1_{1}-}) + d(  Y^3_{ \wt \tau_{B(x,r)} \land  T^1_{1}-},  Y^3_{ \wt \tau_{B(x,r)} \land T^1_{1}}) + d(  Y^3_{ \wt \tau_{B(x,r)}\land T^1_{1}},  Y^3_{ \wt \tau_{B(x,r)}})\nn\\
 	&\le d(x, Y^3_{ \wt \tau_{B(x,r)}\land T^1_{1}-})  + d(  Y^3_{ \wt \tau_{B(x,r)}\land T^1_{1}},  Y^3_{ \wt \tau_{B(x,r)}}) + \frac{r}{4}.
 \end{align}
 In the last inequality above, we used the fact  that the jump size of $Y^3$ at time $T^1_{1}$ is at most $r/4$. Denote by  $\theta^{Y^3}$  the shift operator with respect to $Y^3$. In view of the Meyer's construction, using  the strong Markov property, we obtain from \eqref{e:I3-jump} that
 \begin{align*}
 	I_3 &\le {\mathbb P}^x\Big(d(x, Y^3_{ \wt \tau_{B(x,r)}\land T^1_{1}-})>r/3, \, \wt \tau_{B(x,r)} \le t, \, \wt \tau_{B(x,r)} <T^1_{2} \wedge T^2_{1} \Big)\nn\\
 	&\quad + {\mathbb P}^x\Big(d(  Y^3_{ \wt \tau_{B(x,r)}\land T^1_{1}},  Y^3_{ \wt \tau_{B(x,r)}}) >r/3, \, \wt \tau_{B(x,r)} \le t, \, \wt \tau_{B(x,r)} <T^1_{2} \wedge T^2_{1} \Big) \\
 	&\le {\mathbb P}^x\Big( \tau^{(\rho)}_{B(x, r/3)}  \le  \wt\tau_{B(x,r)} \le t, \;  \wt\tau_{B(x,r)} < T^1_{2} \land T^2_{1} \Big) \nn\\
 	&\quad + {\mathbb P}^x\Big(\tau^{(\rho)}_{B(Y^3_{\wt \tau_{B(x,r)} \land T^1_1}, r/3)}     \circ \theta^{Y^3}_{\wt \tau_{B(x,r)}\land T^1_1} \le  \wt\tau_{B(x,r)} \le t, \;  \wt\tau_{B(x,r)} < T^1_{2} \land T^2_{1}  \Big) \nn\\
 	& \le 2 \sup_{z \in B(x,5r/4)} {\mathbb P}^z \big(\tau^{(\rho)}_{B(z, r/3)}  \le t\big).
 \end{align*}
 In the second inequality above, we used the  fact that $Y^3_{\wt \tau_{B(x,r)}\land T^1_1} \in B(x, r + r/4)$. Therefore, we obtain from Markov inequality, \eqref{e:z-rho} and  Lemma \ref{l:3.9} that
\begin{align*}
	I_3& \le  2 \sup_{z \in B(x, 5r/4)} \E^z \bigg[ \exp\left(\frac{C_1K_1t}{\phi(z,\rho)} -\frac{C_1K_1}{\phi(z,\rho)}\tau_{B(z,r/3)}^{(\rho)}  \right)\bigg]\\
	& \le  2\exp \bigg(\frac{c_3t}{\phi(x,r)}\left(\frac{r}{\rho}\right)^{\beta_2} \bigg) \sup_{z \in B(x, 5r/4)} \E^z \bigg[ \exp\left( -\frac{C_1K_1}{\phi(z,\rho)}\tau_{B(z,r/3)}^{(\rho)}  \right)\bigg]\\
	&\le c_4 \bigg( \frac{\rho}{r}\bigg)^{\theta_1}\exp \bigg(\frac{c_3t}{\phi(x,r)}\left(\frac{r}{\rho}\right)^{\beta_2} \bigg),
\end{align*}
where $\theta_1,C_1, K_1>0$ are the constants in \eqref{e:l.3.9}. 
 Finally, since $t<C_U^{-1}\phi(x,r)$, we deduce from the definition of $\rho$ and  \eqref{e:4.30} that
\begin{align*}
	{\mathbb P}^x(\tau_{B(x,r)} \le t) &\le    \frac{c_1^2\,t^2}{\phi(x,r)^2}   \frac{\phi(x,r)}{C_Ut} +  \frac{c_2t}{\phi(x,r)} +  c_4 \bigg( \frac{C_Ut}{\phi(x,r)}\bigg)^{\theta_1/(2\beta_2)}\exp \bigg(\frac{c_3t}{\phi(x,r)}\left(\frac{\phi(x,r)}{C_Ut}\right)^{1/2} \bigg) \\
	& \le \frac{(c_1^2C_U^{-1}+c_2)t}{\phi(x,r)}  + e^{c_3} c_4 \bigg( \frac{C_Ut}{\phi(x,r)}\bigg)^{\theta_1/(2\beta_2)} \le c_5 \bigg( \frac{t}{\phi(x,r)}\bigg)^{((2\beta_2) \wedge \theta_1)/(2\beta_2)} .
\end{align*} 

(ii) The result follows from 
\cite[Proposition 4.9(ii) and Corollary 4.10]{CKL}.
\qed

An event $G$ is called \textit{shift-invariant}  if $G$ is a tail event (i.e. $\cap_{t>0}^\infty \sigma(X_s:s>t)$-measurable), and ${\mathbb P}^y(G)= {\mathbb P}^y(G \circ \theta_t)$ for all $y \in M$ and $t>0$.

The following zero-one law for shift-invariant events is established in \cite[Proposition 4.15]{CKL}. 

\begin{prop}\label{p:law01}
	Suppose that \VRDi \ holds. If \eqref{B1}, \eqref{B2}, \eqref{B3} and \eqref{B4+} hold, then for every shift-invariant $G$, it holds either ${\mathbb P}^z(G)=0$ for all $z \in M$ or else ${\mathbb P}^z(G) = 1$ for all $z \in M$.
\end{prop}

Now, we are ready to prove our main results in Section \ref{s:intro}.

\bigskip

\noindent \textbf{Proof of Theorem \ref{inf}.}  
In view of Remark \ref{r:a}, it suffices to prove for the case when $\phi(x,r)=\E^x[\tau_{B(x,r)}]$. We claim that there exist constants $q_2 \ge q_1>0$ such that for all $x \in U$, 
 \begin{equation}\label{e:inf0}
  \limsup_{r\to0} \frac{\tau_{B(x,r)}}{\phi(x,r) \log|\log \phi(x,r)|} \in [q_1, q_2],\qquad {\mathbb P}^x\mbox{-a.s.}
 \end{equation}

 We follow the main idea of the proof in \cite[Theorem 3.7]{KKW17}  and will prove upper and lower bound of the  limsup behavior in  \eqref{e:inf0} separately.
 
  Pick $x \in U$.  Let $C_7>0$ be the constant in  \eqref{A4}. We set
  \begin{equation*}
l_n:=\phi^{-1}(x,e^{-n}) \;\; \text{ and } \;\;  A_n:=\Big\{ \sup_{l_{n+1}\le r\le l_n} \frac{\tau_{B(x,r)}}{\phi(x,r) \log|\log \phi(x,r)|}\ge \frac{2e }{C_7} \Big\}, \quad n \ge 3.
 \end{equation*}
Since $\lim_{r \to 0} \phi(x,r)=0$ by \eqref{A2},  we have $\lim_{n \to \infty}l_n=0$. Then using \eqref{A4}, we get that for all $n$ large enough, 
\begin{align*}
	{\mathbb P}^x(A_n) &\le {\mathbb P}^x \bigg( \tau_{B(x,l_n)}\ge \frac{2e}{C_7} \phi(x,l_{n+1}) \log|\log \phi(x,l_{n+1})|  \bigg) \\
	&\le {\mathbb P}^x \bigg( \frac{\tau_{B(x,l_n)}}{\phi(x,l_n)}\ge \frac{2\log n}{C_7}  \bigg) \le C_6 e^{C_7} n^{-2}. 
\end{align*}
Thus, $\sum_{n=3}^\infty {\mathbb P}^x(A_n)<\infty$. Then by the Borel-Cantelli lemma, the upper bound in \eqref{e:inf0} holds with $q_2=2e/C_7$.

 Now, we prove the lower bound in \eqref{e:inf0}. Let $C_1,C_5$ be the constants in \eqref{A1} and \eqref{A4}, and set
\begin{equation*}
 r_n:=\phi^{-1}(x,e^{-n^2}) \;\; \text{ and } \;\; u_n:=\frac{\phi(x,r_{n}) \log|\log \phi(x,r_n)|}{8C_1C_5}, \quad n \ge 3.
\end{equation*}
We also define for $n \ge 3$,
\begin{align*}
E_n&:=\big\{ \sup_{0<s\le u_{n+1}}d(x,X_s) \ge r_n\big \}, \qquad 
F_n:=\big\{   \sup_{u_{n+1} < s\le u_{n}}d(X_{u_{n+1}}, X_s) \ge r_n\big\},\\
G_n&:= \big\{  \sup_{0 < s\le u_{n} } d(x,X_s) \ge 2r_n  \big\}, \qquad H_n:= \cap_{k=n}^{2n} G_k = \Big\{ \sup_{n \le k \le 2n}\frac{\tau_{B(x, 2r_k)}}{u_{k}} \le 1\Big\}.
\end{align*} 
Note that  $G_n \subset E_n \cup F_n$ for all $n \ge 3$ by the triangle inequality. Thus, we have
 \begin{equation}\label{e:decomA}
 	H_n  \subset  \cap_{k=n}^{2n}  \big(  E_k \cup  (F_k \setminus E_{k})  \big) \subset \big(\cup_{k=n}^{2n} E_k \big) \cup  \big( \cap_{k=n}^{2n} (F_k \setminus E_k) \big). 
 \end{equation}
By Proposition \ref{p:EP}(i), we have that for all large enough $k$, 
\begin{align}\label{e:inf1}
{\mathbb P}^x (E_k) &= {\mathbb P}^x( \tau_{B(x, r_k)} \le u_{k+1}) \le  c_1 \bigg( \frac{u_{k+1}}{\phi(x,r_k)} \bigg)^\theta =
  c_2 e^{-2\theta k} (\log (k+1))^\theta.
\end{align} 
Next, by \eqref{A1} and \eqref{A4}, we have that for all $k$ large enough  and  $z \in B(x, r_k)$,
\begin{equation}\label{e:F-1}
{\mathbb P}^z (\tau_{B(z,r_k)} \ge u_{k} ) \ge {\mathbb P}^z\bigg( \tau_{B(z,r_k)}\ge \frac{ \phi(z,r_{k}) \log k }{4C_5}\bigg) \ge 
 C_4e^{-C_5} k^{-1/4} \ge 1-\exp \big( -C_4e^{-C_5}k^{-1/4}\big)
\end{equation}
and hence
\begin{equation}\label{e:F-2}
 {\mathbb P}^z\big(\sup_{0<s\le u_{k}-u_{k+1} } d(z,X_s) \ge r_k 
\big)  \le 1 -  {\mathbb P}^z (\tau_{B(z,r_k)} \ge u_{k} )  \le  \exp \big( -C_4e^{-C_5}k^{-1/4}\big).
\end{equation}
Thus, using the Markov property, we get that for all  $n$ large enough,
\begin{align}\label{e:F-3}
	&{\mathbb P}^x \big(\cap_{k=n}^{2n}  (F_k \setminus E_k)\big) 
	\le 
	\E^x {\mathbb P}^x\big(\cap_{k=n}^{2n} F_k, \, X_{u_{j+1}} \in B(x,r_j), \,\, n \le j \le 2n \, | \, \FF_{u_{n+1}}\big) \nn\\[6pt]
	&  \le {\mathbb P}^x \big(\cap_{k=n+1}^{2n} F_k, \, X_{u_{j+1}} \in B(x,r_j), \,\, n+1 \le j \le 2n \big)   \sup_{z \in B(x, r_n)} {\mathbb P}^z\big(\sup_{0<s\le u_n - u_{n+1} } d(z,X_s)  \ge r_n  \big)  \nn\\
	& \le   \exp \big( -C_4e^{-C_5}n^{-1/4}\big)\,\E^x {\mathbb P}^x \big(\cap_{k=n+1}^{2n} F_k, \, X_{u_{j+1}} \in B(x,r_j), \,\, n+1 \le j \le 2n \, | \, \FF_{u_{n+2}}\big) \nn\\
	& \le...\le  \prod_{k=n}^{2n} \exp \big( -C_4e^{-C_5}k^{-1/4}\big) \le \prod_{k=n}^{2n} \big( -C_4e^{-C_5}(2n)^{-1/4}\big)  \le \exp(\, -  c_3 n^{3/4}).
\end{align}
Therefore,  by combining the above with \eqref{e:decomA} and \eqref{e:inf1}, we get that for all  $n$ large enough,
\begin{align*} 
	{\mathbb P}^x (H_n) &\le \sum_{k=n}^{2n}{\mathbb P}^x(E_k)+ {\mathbb P}^x \left(\cap_{k=n}^{2n} (F_k \setminus E_k)\right) \le c_2e^{-2\theta n} (n+1)(\log (2n+1))^\theta  +   \exp(\, -  c_3 n^{3/4}),
\end{align*}
which yields  $\sum_{n=3}^\infty {\mathbb P}^x(H_n)<\infty$. By the Borel-Cantelli lemma,  it follows that
\begin{equation}\label{e:liminflower}
 {\mathbb P}^x \big(\limsup_{k \to \infty} \frac{\tau_{B(x,2r_k)}}{u_k}\ge 1\big)=1.
\end{equation}
Since $\lim_{k \to \infty} r_k=0$ and 
$$
u_k \ge \frac{\phi(x, 2r_k) \log |\log \phi(x, 2r_k)|}{2^{3+\beta_2}C_1C_5C_U} 
$$
for all $k$ large enough by \eqref{e:scaling-zero}, we conclude from \eqref{e:liminflower} that the lower bound in \eqref{e:inf0} holds.

\smallskip

Now, we claim that for all $x \in U$, it holds that
\begin{equation}\label{e:inf3}
\liminf_{t \to 0} \frac{\phi\big(x,\sup_{0<s \le t}d(x,X_s)\big)}{t/\log |\log t|} \in [e^{-1}q_2^{-1}, q_1^{-1}],\qquad
{\mathbb P}^x\mbox{-a.s.}
\end{equation}
Note that once we prove \eqref{e:inf3}, the proof is finished thanks to the Blumenthal's zero-one law. Also, since $q_1$ and $q_2$ in \eqref{e:inf3} can be chosen by $C$ and the constants $q_1$ and $q_2$ with respect to $\E^x[\tau_{B(x,r)}]$, Remark \ref{r:a} is also verified.

Here, we show \eqref{e:inf3}. Recall that $l_n:=\phi^{-1}(x,e^{-n})$. Set $t_n:=\phi(x, l_n) \log|\log \phi(x, l_n)|=e^{-n}\log n$. Choose any $\delta>0$.
By \eqref{e:inf0}, for ${\mathbb P}^x$-a.s. $\omega$, there exists $N =N(\omega)$ such that $
\tau_{B(x,l_n)} \le  (q_2+\delta) t_n$  for all $n \ge N$. Thus, by \eqref{e:scaling-zero}, it holds that for  ${\mathbb P}^x$-a.s. $\omega$,
\begin{align*}
\liminf_{t \to 0} \frac{\phi\big(x,\sup_{0<s \le t}d(x,X_s)\big)}{ t/\log |\log t|} &\ge \liminf_{n \to \infty} \inf_{  t \in [(q_2+\delta)t_n, (q_2+\delta)t_{n-1}]} \frac{\phi\big(x,\sup_{0<s \le t}d(x,X_s)\big)}{ t/\log |\log t|}\\
&\ge \liminf_{n \to \infty}\frac{\phi\big(x,\sup_{0<s \le (q_2+\delta)t_n}d(x,X_s)\big)}{ (q_2+\delta)t_{n-1}/ \log |\log (q_2+\delta)t_{n-1}|}\\
& \ge \liminf_{n \to \infty} \frac{\phi(x, l_n)}{ (q_2+\delta)e^{-(n-1)} \log(n-1) / \log n } =  \frac{1}{e(q_2+\delta)}.
\end{align*}
On the other hand, we also get from  \eqref{e:inf0} that for ${\mathbb P}^x$-a.s. $\omega$, there exists a decreasing sequence $(\wt r_n)_{n \ge 1} = (\wt r_n(\omega))_{n \ge 1}$ converging to zero such that
\begin{equation*}
\tau_{B(x,\wt r_n)}(\omega) \ge (q_1-\delta) \phi(x, \wt r_n) \log|\log \phi(x, \wt r_n)|=:\wt t_{\delta, n} \quad \text{for all} \;\; n \ge 1.
\end{equation*}
It follows that ${\mathbb P}^x$-a.s.,
\begin{align*}
\liminf_{t \to 0} \frac{\phi\big(x,\sup_{0<s \le t}d(x,X_s)\big)}{ t/\log |\log t|} &\le \liminf_{n \to \infty}  \frac{\phi\big(x,\sup_{0<s \le \wt t_{\delta, n}}d(x,X_s)\big)}{\wt t_{\delta, n}/\log |\log \wt t_{\delta, n}|} \\
& \le \liminf_{n \to \infty}  \frac{\phi(x, \wt r_n)}{\wt t_{\delta, n}/\log |\log \wt t_{\delta, n}|}\\
&= \liminf_{n \to \infty}  \frac{\phi(x, \wt r_n)}{ (q_1-\delta) \phi(x, \wt r_n)} \frac{\log |\log \wt t_{\delta,n}| }{\log |\log \phi(x, \wt r_n)|}=\frac{1}{q_1-\delta}.
\end{align*}
Since $\delta$ can be arbitrarily small, we obtain  \eqref{e:inf3}. The proof is complete. \qed

\noindent \textbf{Proof of Corollary \ref{c:inf}.} Using \eqref{e:scaling-zero} and \eqref{e:scaling-zero-lower}, we can see from  \eqref{e:inf3} that there exist constants $c_2 \ge c_1>0$ such that for all $x \in U$, $\liminf_{t \to 0} \sup_{0<s \le t}d(x,X_s)/\phi^{-1}(x,t/\log |\log t|) \in [c_1, c_2]$, ${\mathbb P}^x$-a.s. 
Then using the Blumenthal's zero-one law again, we obtain the result. \qed

\noindent \textbf{Proof of Theorem \ref{t:infinf}.} By \eqref{e:phi}, it suffices to prove the theorem with $\phi(r):=\E^o[\tau_{B(o,r)}]$.  We follow the proof of Theorem \ref{inf} with some modifications. To obtain the desired result,  by repeating the arguments for obtaining \eqref{e:inf3} and using \eqref{e:scaling-infty}, it is enough to show that  there exist constants $q_4 \ge q_3>0$ such that for all $x,y \in M$,
 \begin{equation}\label{e:inf0'}
\limsup_{r\to \infty} \frac{\tau_{B(x,r)}}{\phi(r) \log\log \phi(r)} \in [q_3, q_4], \qquad 
{\mathbb P}^y\mbox{-a.s.}
\end{equation}
By \eqref{e:scaling-infty} and the monotone property of $\phi(r)$, we have that, for all $x,y \in M$, since $d(x,y)<\infty$,
 \begin{equation*}
\limsup_{r\to \infty} \frac{\tau_{B(x,r)}}{\phi(r) \log\log \phi(r)} \le \limsup_{r\to \infty} \frac{\tau_{B(y,r+d(x,y))}}{\phi(r) \log\log \phi(r)} \le 2^{\beta_2}  C_U\limsup_{r\to \infty} \frac{\tau_{B(y, 2r)}}{\phi(2r) \log\log \phi(2r)} 
\end{equation*}
and
 \begin{equation*}
\limsup_{r\to \infty} \frac{\tau_{B(x,r)}}{\phi(r) \log\log \phi(r)}  \ge \limsup_{r\to \infty} \frac{\tau_{B(y,r-d(x,y))}}{\phi(r) \log\log \phi(r)}  \ge 2^{-\beta_2} C_U^{-1}\limsup_{r\to \infty} \frac{\tau_{B(y, r/2)}}{\phi(r/2) \log\log \phi(r/2)}.
\end{equation*}
Thus, to get \eqref{e:inf0'}, it is enough to prove that  for all $y \in M$,
 \begin{equation}\label{e:inf0''}
\limsup_{r\to \infty} \frac{\tau_{B(y,r)}}{\phi(r) \log\log \phi(r)} \in [2^{\beta_2}C_Uq_3, \, 2^{-\beta_2}C_U^{-1}q_4],\qquad
{\mathbb P}^y\mbox{-a.s.}
\end{equation}

Let $y \in M$.   With the constant $C_7$ in  \eqref{B4}, we define
\begin{equation*}
\wt	l_n=\phi^{-1}(e^{n}) \;\; \text{ and } \;\; \wt A_n=\Big\{ \sup_{\wt l_{n}\le r\le \wt l_{n+1}} \frac{\tau_{B(y,r)}}{\phi(r) \log\log \phi(r)}\ge \frac{2e }{C_7} \Big\}, \quad n \ge 3.
\end{equation*}
Note that $\lim_{n \to \infty}\wt l_n=\infty$ by \eqref{B2} (see Remark \ref{r:basic}(iii)). Hence, $\wt l_n>R_\infty {\mathsf{d}}(y)^\ups$ for all $n$ large enough. Then by \eqref{B4}, we get that for all $n$ large enough, 
\begin{align*}
	{\mathbb P}^y(\wt A_n) &\le {\mathbb P}^y \bigg( \tau_{B(y,\wt l_{n+1})}\ge \frac{2e}{C_7} \phi(\wt l_{n}) \log\log \phi(\wt l_{n})  \bigg) ={\mathbb P}^y \bigg( \frac{\tau_{B(y,\wt l_{n+1})}}{\phi(\wt l_{n+1})}\ge \frac{2\log n}{C_7}  \bigg) \le C_6 e^{C_7} n^{-2}.
\end{align*}
Using the Borel-Cantelli lemma, we deduce that the upper bound in \eqref{e:inf0''} holds true.

To prove the lower bound, we set 
$$
m_n:=\phi^{-1}(e^{n^2}) \;\; \text{ and } \;\; s_n:= \frac{\phi(m_n) \log \log \phi(m_n)}{8C_1C_5},  \quad n \ge 3,
$$
where $C_1,C_5$ are the constants in \eqref{B1} and \eqref{B4}.  We also let
\begin{align*}
\wt E_n&: =\big\{\sup_{0 < s\le s_{n-1} } d(y,X_s) \ge m_n\big\}, \qquad \wt F_n:=\big\{\sup_{s_{n-1} < s\le s_{n}} d(X_{s_{n-1}}, X_s) \ge  m_n \big\},\\
\wt G_n&:= \big\{  \sup_{0 < s \le s_{n} } d(y,X_s) \ge 2 m_n  \big\}, \qquad \wt H_n:= \cap_{k=n}^{2n} \wt G_k = \Big\{ \sup_{n \le k \le 2n}\frac{\tau_{B(y, 2m_k)}}{s_{k}} \le 1\Big\}. 
\end{align*}
Then for all $n$,  $\wt G_n \subset \wt E_n \cup \wt F_n$ by the triangle inequality so that $\wt H_n \subset (\cup_{k=n}^{2n} \wt E_k) \cup  (\cap_{k=n}^{2n} (\wt F_k \setminus \wt E_k))$.

First, using Proposition \ref{p:EP}(ii) (with $\ups_1 = \sqrt \ups<1$) and \eqref{e:scaling-infty} twice,  we get that for all $n$ large enough, 
\begin{align}\label{e:wtF}
&{\mathbb P}^y (\wt E_n) \le  {\mathbb P}^x\big( \tau_{B(x,  m_n)} \le s_{n-1} + \phi(2  m_n^{\sqrt\ups})\big)\le  c_1\frac{s_{n-1} + \phi(2 m_n^{\sqrt\ups})}{\phi(m_n)}\nn\\
& \le  c_2 e^{-2n} \log n  + c_2 m_n^{-(1-\sqrt\ups)\beta_1}\le  c_2 e^{-2n} \log n  + c_3 R_\infty \bigg(\frac{e^{n^2}}{\phi(2R_\infty)} \bigg)^{-(1-\sqrt \ups) \beta_1/\beta_2} \le c_4 e^{-n}.
\end{align}

Next, we note that  since  $\ups<1$ and $\lim_{n \to \infty}m_n=\infty$,  for all $n$ large enough and $z \in B(y, m_n)$,  
$$
R_\infty {\mathsf{d}}(z)^\ups \le R_\infty {\mathsf{d}}(y)^\ups + R_\infty d(y,z)^\ups < m_n/2 + R_\infty m_n^{\ups} <m_n.
$$ 
Hence, by following the calculations \eqref{e:F-1}, \eqref{e:F-2} and \eqref{e:F-3}, using \eqref{B1}, \eqref{B4} and the Markov property, we get that for all $n$ large enough,
\begin{align*}
&{\mathbb P}^y \big(\cap_{k=n}^{2n}  (\wt F_k \setminus \wt E_{k})\big)  \le {\mathbb P}^x \big(\cap_{k=n}^{2n} \wt F_k, \,  X_{s_{j-1}} \in B(y, m_{j}), n \le j \le 2n \big) \nn\\[6pt]
&  \le {\mathbb P}^y \big(\cap_{k=n}^{2n-1} \wt F_k, \,   X_{s_{j-1}} \in B(y, m_{j}), n \le j \le 2n-1 \big)\hskip-.1in  \sup_{z \in B(y, m_{2n})}\hskip-.1in{\mathbb P}^z\hskip-.02in\big( \hskip-.1in\sup_{0<s\le s_{2n} - s_{2n-1} } d(z,X_s) \ge m_{2n}  \big)  \nn\\
& \le...\le  \prod_{k=n}^{2n} \exp(-C_4 e^{-C_5}k^{-1/4})  \le \exp(-c_5n^{3/4}).
\end{align*}

By combining the above with \eqref{e:wtF}, we get  
$$\sum_{n=1}^\infty {\mathbb P}^y(\wt H_n) \le \sum_{n=1}^\infty (\sum_{k=n}^{2n} {\mathbb P}^y(\wt E_k)  +{\mathbb P}^y (\cap_{k=n}^{2n} (\wt F_k \setminus \wt E_k)))<\infty.$$
 Hence  ${\mathbb P}^y(\limsup \wt H_n)=0$ by  the Borel-Cantelli lemma. Since  $\lim_{k \to \infty}m_k=\infty$ and
$$s_k \ge \frac{\phi(2m_k) \log \log \phi(2m_k)}{2^{4 + \beta_2}C_1C_5C_U}$$
for all $k$ large enough by \eqref{e:scaling-infty}, we get the lower bound in \eqref{e:inf0''}. The proof is complete.  \qed

\noindent \textbf{Proof of Corollary \ref{c:infinf}.} 
By Proposition \ref{r:E}(ii) and Theorem \ref{t:infinf}, the  liminf law \eqref{20} holds under the current setting. Thus, by   Proposition \ref{p:law01}, it suffices to show that for every $x \in M$ and $\lambda>0$, 
$$E=E(x, \lambda) :=\Big\{\liminf_{t \to \infty} \frac{\phi\big(\sup_{0 < s \le t} d(x,X_s)\big)}{t/\log \log t} \ge \lambda  \Big\}$$ is  a shift-invariant event. 

Let $\lambda, u>0$ and $x,y \in M$.   Observe that by the Markov property,
$$
E \circ \theta_u=\Big\{\liminf_{t \to \infty} \frac{\phi\big(\sup_{0 < s \le t} d(x,X_{s+u})\big)}{t/\log \log t} \ge \lambda  \Big\}.$$
Since $X$ is conservative by Proposition \ref{p:EP}(ii), for all $t>0$, it holds that $\sup_{0<s\le t}d(x,X_s) < \infty$, ${\mathbb P}^y$-a.s. Hence, since $\phi$ is positive, we see that for all $t>0$,
\begin{equation*}
\phi\big(\sup_{0 < s \le t+u} d(x,X_{s}) \big) = \phi\big(\sup_{s \in (u, t+u] \, \cup (0, u]} d(x,X_{s}) \big) \le \phi\big(\sup_{0 < s \le t} d(x,X_{s+u})\big) + \phi\big(\sup_{0<s \le u}d(x, X_s) \big).
\end{equation*}
Therefore, we get that for ${\mathbb P}^y$-a.s. $\omega \in E$,
\begin{align*}
&\liminf_{t \to \infty} \frac{\phi\big(\sup_{0 < s \le t} d(x,X_{s+u})\big)}{t/\log \log t} \\
&\ge \liminf_{t \to \infty} \frac{\phi\big(\sup_{0 < s \le t+u} d(x,X_{s}) \big)}{(t+u)/\log \log (t+u)}\frac{ (t+u)/\log \log (t+u)}{t/\log \log t} - \limsup_{t \to \infty} \frac{\phi\big(\sup_{0<s \le u}d(x, X_s)\big)}{t/\log \log t}\ge \lambda.
\end{align*}
On the other hand, for every  $\omega \in E \circ \theta_u$, we see that
\begin{align*}
&\liminf_{t \to \infty} \frac{\phi\big(\sup_{0 < s \le t} d(x,X_{s})\big)}{t/\log \log t} = \liminf_{t \to \infty} \frac{\phi\big(\sup_{0 < s \le t+u} d(x,X_{s})\big)}{(t+u)/\log \log (t+u)}\\
&\ge \liminf_{t \to \infty} \frac{\phi\big(\sup_{0 < s \le t} d(x,X_{s+u})  \big)}{t/ \log \log t}\frac{t/\log \log t} { (t+u)/\log \log (t+u)}\ge \lambda.
\end{align*}
Hence, ${\mathbb P}^y(E_u) \le {\mathbb P}^y(E)$. Since $E$ is clearly a tail event,  this completes the proof.
\qed

\section{Appendix}\label{s:A}
In this section, we 
follow the setting in Section \ref{s:intro} and
compare the conditions in this paper with those in \cite{CKL}. We  recall the conditions ${\rm Tail}$  and ${\rm NDL}$, and upper and lower scaling properties for nonnegative functions which were presented in \cite[Definitions 1.5, 1.6 and 1.9]{CKL}. We will give a sufficient condition for NDL too. 

\smallskip

Throughout the appendix, we let $\varphi: (0,\infty) \to (0,\infty)$ be an increasing and continuous function such that $\displaystyle \lim_{r \to 0} \varphi(r) = 0$ and  $\displaystyle \lim_{r \to \infty} \varphi(r) = \infty$.

\begin{definition}\label{d:0}
	{\rm Let  $R_0 \in (0,\infty]$ be a constant and $U \subset M$ be an open set.
		
		\smallskip
		
		\noindent	(i)  We say that  ${\mathrm {Tail}}_{R_0}(\varphi, U)$ holds if there exist  constants $C_0\in (0,1)$, $c_J>1$  such that for all $x \in U$ and  $0<r<R_0 \wedge (C_0\updelta_U(x))$,
		\begin{equation}\label{e:Tail_0}
			\frac{c_J^{-1}}{\varphi(r)} \le J(x,M_\partial \setminus B(x,r)) \le \frac{c_J}{\varphi(r)}.
		\end{equation}
		We say that ${\mathrm {Tail}}_{R_0}(\varphi, U, \le)$  (resp.  ${\mathrm {Tail}}_{R_0}(\varphi, U, \ge)$) holds (with $C_0$) if the upper bound (resp. lower bound) in \eqref{e:Tail_0} holds for all $x \in U$ and $0<r<R_0 \wedge (C_0\updelta_U(x))$.
		
		\smallskip
		
		\noindent (ii) We say that ${\mathrm {E}}_{R_0}(\varphi, U)$ holds  if there exist constants $C_0 \in (0,1)$, $C_1>0$ and $c_E > 1$ such that for all $x \in U$ and $0<r<R_0 \land (C_0\updelta_U(x))$,	\begin{equation}\label{e:Eo}		c_E^{-1}\varphi(C_1 r) \le \E^x[\tau_{B(x,r)}]  \le c_E \varphi(C_1 r).	\end{equation}

		\noindent (iii)  We say that ${\mathrm {NDL}}_{R_0}(\varphi, U)$ holds if there exist constants $C_2, \eta \in (0,1)$ and  $c_{l}>0$ such that for all $x \in U$ and $0<r<R_0 \land (C_2\updelta_U(x))$, the heat kernel $p^{B(x,r)}(t,y,z)$ of $X^{B(x,r)}$ exists and 
		\begin{equation}\label{e:NDL_inf}
			p^{B(x,r)}(\varphi(\eta r),y,z) \ge \frac{c_{l}}{V(x, r)}, \quad \quad y,z \in B(x, \eta^2r).
		\end{equation} 
	}
\end{definition}

\begin{definition}\label{d:inf}
	{\rm Let $R_\infty \ge 1$ and $\ups \in(0,1)$ be constants.
		
		\smallskip
		
		\noindent  (i)  We say that  \Taili \ holds if there exists a constant $c_J>1$ such that \eqref{e:Tail_0} holds for all $x \in M$ and $r >R_\infty \ep(x)^\ups$. We say that \Tailil \ (resp. \Tailig) holds if the upper bound (resp. lower bound) in \eqref{e:Tail_0} holds for all $x \in M$ and $r>R_\infty \ep(x)^\ups$.
		
		\smallskip

		\noindent (ii)  We say that \Ei \ holds if there exist constants $\ups \in (0,1)$, $C_1 >0$ and  $c_E > 1$ such that \eqref{e:Eo} holds for all  $x \in M$ and $r>R_\infty \ep(x)^\ups$.
		
		\smallskip
		
		\noindent (iii)  We say that \NDLi  \ holds if there exist
		constants $\eta \in (0,1)$ and $c_l>0$ such that for all $x \in M$ and  $r >R_\infty \ep(x)^\ups $,  the heat kernel $p^{B(x,r)}(t,y,z)$ of $X^{B(x,r)}$ exists and satisfies \eqref{e:NDL_inf}.}
\end{definition}

\begin{defn}\label{d:ws}
	{\rm For $g:(0,\infty) \to (0,\infty)$ and constants $a \in (0, \infty]$,  $\beta_1, \beta_2>0$, $ c_1,c_2>0$, we say that ${\mathrm{L}}_a(g,\beta_1, c_1)$ (resp. ${\mathrm{L}}^a(g,\beta_1, c_1)$) holds if
		$$ \frac{g(r)}{g(s)} \geq c_1 \Big(\frac{r}{s}\Big)^{\beta_1} \quad \text{for all} \quad s\leq r< a\;\;(\text{resp.}\,\;a < s\leq r).$$
		and	we say that $\U_a(g,\beta_2, c_2)$ (resp. $\U^a(g,\beta_2, c_2)$) holds if
		$$ \frac{g(r)}{g(s)} \leq c_2 \Big(\frac{r}{s}\Big)^{\beta_2} \quad \text{for all} \quad s\leq r< a\;\;(\text{resp.}\;a < s\leq r).$$
		We say that  ${\mathrm{L}}(g,\beta_1,c_1)$ holds if  ${\mathrm{L}}_\infty(g,\beta_1,c_1)$ holds, and that $\U(g,\beta_2,c_2)$ holds if  $\U_\infty(g,\beta_2,c_2)$ holds.
}\end{defn}

We now show that the assumptions in this papers are weaker than those in \cite{CKL}.

\begin{lem}\label{l:NDL}
	(i) Suppose that  ${\rm VRD}_{R_0}(U)$, ${\rm Tail}_{R_0}(\varphi,U,\le)$, ${\rm U}_{R_0}(\varphi,\beta_2,C_U)$ and ${\rm NDL}_{R_0}(\varphi,U)$ hold. Then the function $\phi(x,r):=\varphi(r)$ satisfies \eqref{e:phi} for all $x \in U$ and $0<r<r_0 \wedge (C_0'\updelta_U(x))$ with some constants $r_0>0$ and $C_0'\in(0,1)$, and conditions	 \eqref{A1}, \eqref{A2}, \eqref{A3} and \eqref{A4+} hold for $U$.

\noindent	(ii) Suppose that ${\rm VRD}^{R_\infty}(\ups)$, ${\rm Tail}^{R_\infty}(\varphi,\ups, \le)$, ${\rm U}^{R_\infty}(\varphi,\beta_2,C_U)$, ${\rm L}^{R_\infty}(\varphi,\beta_1,C_L)$ and ${\rm NDL}^{R_\infty}(\varphi,\ups)$ hold. Then the function $\phi(x,r):=\varphi(r)$ satisfies \eqref{e:phi} for all $x \in M$ and $r>r_1 {\mathsf{d}}(x)^\ups$ with some constant $r_1\ge1$, and conditions \eqref{B1}, \eqref{B2}, \eqref{B3} and \eqref{B4+} hold.

\end{lem}
\begin{proof} (i) Under the setting, by \cite[Proposition 4.3(i)]{CKL} and ${\rm U}_{R_0}(\varphi,\beta_2,C_U)$, there exist constants  $r_0\in(0,R_0)$, $C_0'\in(0,1)$ and $c_1>1$ such that $\E^x[\tau_{B(x,r)}] \asymp \varphi(r)$ for $x \in U$ and $0<r < r_0 \land C_0' \updelta_U(x)$. Hence, using  ${\rm U}_{R_0}(\varphi,\beta_2,C_U)$ and the fact that $\lim_{r\to 0}\varphi(r) =0$, we see that \eqref{A1}-\eqref{A3} hold for $U$. Now \eqref{A4+} immediately follows from  ${\rm NDL}_{R_0}(\varphi,U)$.

\noindent (ii) Similarly, using \cite[Proposition 4.3(ii)]{CKL}, one can deduce the desired results. \end{proof}

Recall the notion of the heat kernel from Section \ref{s:intro}. In the next lemma, we let 
 $X$ be a strong Markov process on $M$ having the heat kernel $p(t,x,y):=p^M(t,x,y)$ such that $p(t,x,y)<\infty$ unless $x = y$. Then by the strong Markov property of $X$, one can see that for any  open set $D \subset M$, the heat kernel $p^D(t,x,y)$ of $X^D$ 
exists and can be written as
\begin{equation}\label{e:dhk}
	p^D(t,x,y) = p(t,x,y) - \E^x \Big[ \E^{X_{\tau_D}} \big[ p(t-\tau_D,X_{\tau_D},y) ; \tau_D < t \big] \Big]. 
\end{equation}
Using \eqref{e:dhk}, the proof of the next lemma is a simple modification of that of \cite[Proposition 2.3]{CKK09} and \cite[Proposition 2.5]{CKSV}. We give a full proof for the reader's convenience.

\begin{lem}\label{l:hkendl}
Let $U \subset M$ be an open subset.	Suppose that   there exist  constants  $R_0 \in (0,\infty]$, $C,C' \ge 1$  such that  ${\mathrm {VRD}}_{R_0}(U)$  holds, and for all $t \in (0,\varphi(R_0/2))$,
	\begin{equation}\label{e:uhk}
		p(t,x,y) \le  \frac{Ct}{V(y,d(x,y)) \varphi(d(x,y))} \quad \mbox{ for all} \;\; x \in M, \, y \in U \, \mbox{ with } \, d(x,y) > C'\varphi^{-1}(t)
	\end{equation}
	and 
	\begin{equation}\label{e:lhk}
		p(t,x,y) \ge \frac{C^{-1}}{V(x,\varphi^{-1}(t))} \quad \mbox{ for all} \;\; x,y \in U \, \mbox{ with } \, d(x,y) < C'^{-1}\varphi^{-1}( t).
	\end{equation}
	Then ${\mathrm {NDL}}_{R_0}(\varphi, U)$ holds true.
\end{lem}
\begin{proof} Set  $\eta:=(2C')^{-1}( 2^{d_2+1} C^2 C_\mu /c_\mu)^{-1/d_1} \in (0,1/2)$ where $d_1, d_2, c_\mu, C_\mu$ are the constants from \eqref{e:VRD}. Choose any $x \in U$, $0<r<R_0 \land (C_V \updelta_U(x))$ and $y,z \in B(x,\eta^2r)$.

We observe that  $B(x, \eta^2r)  \subset B(x, \updelta_U(x)) \subset  U$ and $d(y,z) \le 2\eta^2 r <C'^{-1} \eta r$.  Thus, by \eqref{e:lhk} and  ${\mathrm {VRD}}_{R_0}(U)$, since $\eta<1/2$, it holds that 
\begin{equation}\label{e:hkendl2}
	p(\varphi(\eta r),y,z) \ge \frac{C^{-1}}{V(y,\eta r)} \ge \frac{C^{-1}}{V(x,\eta r + d(x,y))} \ge \frac{C^{-1}}{V(x,2\eta r )} \ge \frac{C^{-1} c_\mu (2\eta)^{-d_1}}{V(x,r)} \ge \frac{2^{d_2+1}C C_\mu }{V(x,r)}.
\end{equation}
On the other hand,  for every $w \in M \setminus B(x,r)$, we see that $d(w,z) \ge d(w,x) - d(x,z) \ge 3r/4  >C' \eta r$.
Therefore,  for every $0<s\le \varphi(\eta r)$ and $w \in M \setminus B(x,r)$, since $\varphi$ is increasing and $\eta<1/2$, we get from \eqref{e:uhk} and $\mathrm{VRD}_{R_0}(U)$ that 
\begin{align}\label{e:hkendl1}
	p(s, w,z) &\le \frac{C \varphi(\eta r)}{V(z,d(w,z))\varphi(d(w,z))}  \le \frac{C \varphi(\eta r)}{V(z,3r/4)\varphi(3r/4)} \le \frac{C}{V(z,3r/4)}   \nn\\[2pt]
	&\le  \frac{C}{V(x,3r/4-d(x,z))} \le  \frac{C}{V(x,r/2)} \le \frac{2^{d_2}CC_\mu }{V(x,r)}.
\end{align}
Therefore, since $X_{\tau_{B(x,r)}} \in M_\partial \setminus B(x,r)$, using the formula \eqref{e:dhk}, we conclude from  \eqref{e:hkendl2} and \eqref{e:hkendl1} that $p^{B(x,r)}(\varphi(\eta r),y,z) \ge 2^{d_2}CC_\mu/V(x,r)$. The proof is complete.  \end{proof}


\small
\begin{thebibliography}{10}


\bibitem{ADS13}
F.~Aurzada, L.~D\"{o}ring, M.~Savov.
\newblock Small time Chung-type LIL for Lévy processes.
\newblock {\em Bernoulli}  19 (2013), no. 1, 115-136.


\bibitem{BGK09}
M.~T. Barlow, A.~Grigor'yan,  T.~Kumagai.
\newblock Heat kernel upper bounds for jump processes and the first exit time.
\newblock {\em J. Reine Angew. Math.} 626 (2009), 135-157.


\bibitem{BJ73}
A.~Benveniste, J.~Jacod.
\newblock Syst\`emes de {L}\'{e}vy des processus de {M}arkov.
\newblock {\em Invent. Math.} 21 (1973), 183-198.


\bibitem{BSW}
B. Böttcher, R. L. Schilling, J. Wang.
\newblock Lévy matters. III. Lévy-type processes: construction, approximation and sample path properties.
\newblock Lecture Notes in Mathematics, 2099. Lévy Matters. Springer, Cham, 2013. 


\bibitem{BM}
B.~Buchmann, R.~Maller.
\newblock  The small-time Chung-Wichura law for Lévy processes with non-vanishing Brownian component.
\newblock {\em 	Probab. Theory Related Fields}  149 (2011), no. 1-2, 303-330.
  
  
  
  	\bibitem{CKW18}
  X.~Chen, T.~Kumagai, J.~Wang.
  \newblock Random conductance models with stable-like jumps: quenched invariance principle.
  \newblock {\em Ann. Appl. Probab.}  31 (2021), no. 3, 1180-1231.
  
  
  
  \bibitem{CKW20-1}
  X.~Chen, T.~Kumagai, J.~Wang.
  \newblock Random conductance models with stable-like jumps: heat kernel estimates and Harnack inequalities. 
  \newblock {\em J. Funct. Anal.}  279 (2020), no. 7, 108656, 51 pp.
  

  
  
  \bibitem{CF}
  Z.-Q. Chen,  M. Fukushima.
  \newblock Symmetric Markov processes, time change, and boundary theory. 
  \newblock London Mathematical Society Monographs Series, 35. Princeton University Press,  2012.
  
  
  
  \bibitem{CKK09}
  Z.-Q. Chen, P.  Kim,  T. Kumagai.
  \newblock  On heat kernel estimates and parabolic Harnack inequality
  for jump processes on metric measure spaces.
  \textit{Acta Math. Sin. (Engl. Ser.)} 25 (2009), no. 7, 1067-1086.
  
  
  		\bibitem{CKW16b}
		Z.-Q. Chen, T.~Kumagai, J.~Wang.
		\newblock Stability of parabolic Harnack inequalities for symmetric non-local
		Dirichlet forms.
		\newblock {\em J. Eur. Math. Soc.}  22 (2020), no. 11, 3747-3803.
		  		\bibitem{CKW19}
		Z.-Q.~Chen, T.~Kumagai, J.~Wang.
		\newblock Heat kernel estimates and parabolic {H}arnack inequalities for
		symmetric {D}irichlet forms.
		\newblock {\em Adv. Math.}  374 (2020), 107269, 71 pp.
  


\bibitem{CKL}
S.~Cho, P.~Kim, J. Lee.
\newblock 	General Law of iterated logarithm for Markov processes: Limsup law.  arXiv:2102.01917v2 [math.PR].


\bibitem{CKSV} S. Cho, P. Kim, R. Song, R., Z. Vondra\v{c}ek.
\newblock Factorization and estimates of Dirichlet heat kernels for non-local operators with critical killings. 
\newblock {\em J. Math. Pures Appl.}  (9) 143 (2020), 208-256.



\bibitem{Ch48}
K.-L. Chung.
\newblock On the maximum partial sums of sequences of independent random variables.
\newblock {\em Trans. Amer. Math. Soc.}  64 (1948), 205-233.


\bibitem{Co65} P.~Courr\`ege.
\newblock Sur la forme int\'egro-diff\'erentielle des op\'erateurs de $C^\infty _k$ dans $C$ satisfaisant au principe du maximum. \newblock {\em S\'eminaire Brelot-Choquet-Deny. Th\'eorie du potentiel.} 10 (1965), no. 1, 1-38.



\bibitem{Du74}
C.~Dupuis.
\newblock Mesure de Hausdorff de la trajectoire de certains processus \`a
  accroissements ind\'{e}pendants et stationnaires.
\newblock In {\em  Séminaire de Probabilités, VIII},  pp. 37-77. Lecture Notes in Math. 381, Springer, Berlin, 1974. 


\bibitem{EM94}
U.~Einmahl, D.~M. Mason.
\newblock A universal Chung-type law of the iterated logarithm.
\newblock {\em Ann. Probab.}  22 (1994), no. 4, 1803-1825.




\bibitem{FP71}
B.~Fristedt, W.~E. Pruitt.
\newblock Lower functions for increasing random walks and subordinators.
\newblock {\em Z. Wahrsch. und Verw. Gebiete}  18 (1971), 167-182.




\bibitem{GRT19}
T.~Grzywny, M.~Ryznar,  B.~Trojan.
\newblock Asymptotic behaviour and estimates of slowly varying convolution semigroups.
\newblock {\em Int. Math. Res. Not.} IMRN 2019, no. 23, 7193-7258.



\bibitem{GS19}
T.~Grzywny, K.~Szczypkowski.
\newblock Heat kernels of non-symmetric {L}\'{e}vy-type operators.
\newblock {\em J. Differential Equations} 267 (2019), no. 10, 6004-6064.


\bibitem{GS21}
T.~Grzywny, K.~Szczypkowski.
\newblock Estimates of heat kernels of non-symmetric {L}\'{e}vy processes.
\newblock {\em Forum Mathematicum} 33(5) (2021), no. 5, 1207-1236.




\bibitem{Hoh}
W. Hoh.
\newblock Pseudo differential operators generating Markov processes. Diss. Habilitationsschrift, Universität Bielefeld, 1998.



\bibitem{JP75}
N.~C. Jain and W.~E. Pruitt.
\newblock The other law of the iterated logarithm.
\newblock {\em Ann. Probab.}  3 (1975), no. 6, 1046-1049.



\bibitem{Ke97}
H.~Kesten.
\newblock A universal form of the {C}hung-type law of the iterated logarithm.
\newblock {\em Ann. Probab.}, 25 (1997), no. 4, 1588--1620.

	\bibitem{KKW17}
P.~Kim, T.~Kumagai,  J.~Wang.
\newblock Laws of the iterated logarithm for symmetric jump processes.
\newblock {\em Bernoulli}  23 (2017), no. 4A, 2330-2379.



\bibitem{KKS} V. Knopova, A. Kulik, Alexei, R.L. Schilling.
\newblock Construction and heat kernel estimates of general stable-like Markov processes.
\newblock{\em Dissertationes Math.} 569 (2021), 86 pp.




\bibitem{KS}
V.~Knopova, R.~L. Schilling.
\newblock On the small-time behaviour of {L}\'{e}vy-type processes.
\newblock {\em Stochastic Process. Appl.} 124 (2014), no. 6, 2249-2265.



\bibitem{Ku}
F. Kühn.
\newblock Lévy matters. VI. Lévy-type processes: moments, construction and heat kernel estimates. 
\newblock Lecture Notes in Mathematics, 2187. Lévy Matters. Springer, Cham, 2017. 

\bibitem{Me75}
P.~A. Meyer.
\newblock Renaissance, recollements, m\'{e}langes, ralentissement de processus
de {M}arkov.
\newblock {\em Ann. Inst. Fourier (Grenoble)}  25 (1975), no. 3-4, xxiii, 465-497.



\bibitem{Pr81}  W. E. Pruitt.
\newblock The growth of random walks and {L}\'{e}vy processes. 
\newblock {\em Ann. Probab.} 9 (1981), no. 6, 948-956.


\bibitem{Sa13}
K.-I. Sato.
\newblock {L\'{e}vy processes and infinitely divisible distributions},
volume~68 of {\em Cambridge Studies in Advanced Mathematics}.
\newblock Cambridge University Press, Cambridge, 2013.




\bibitem{SW13}
R. L. Schilling, J. Wang.
\newblock	Some theorems on Feller processes: transience, local times and ultracontractivity.
\newblock {\em 	Trans. Amer. Math. Soc.} 365 (2013), no. 6, 3255-3286.



\bibitem{T}
S.~J. Taylor.
\newblock Sample path properties of a transient stable process.
\newblock {\em J. Math. Mech.}  16 (1967) 1229-1246.



\bibitem{Wee88}
I.~S. Wee.
\newblock Lower functions for processes with stationary independent increments.
\newblock {\em Probab. Theory Related Fields}  77 (1988), no. 4, 551-566.


\bibitem{Wi74}
M.~J. Wichura.
\newblock On the functional form of the law of the iterated logarithm for the
  partial maxima of independent identically distributed random variables.
\newblock {\em Ann. Probab.}  2 (1974), 202-230.

\bibitem{Xu13}
F.~Xu.
\newblock A class of singular symmetric {M}arkov processes.
\newblock {\em Potential Anal.}, 38 (2013), no. 1, 207-232.

\end{thebibliography}
\end{document}